\documentclass[lettersize,journal]{IEEEtran}
\usepackage{amsmath,amsfonts,amssymb,amsthm}
\usepackage{algorithmic}
\usepackage{algorithm}
\usepackage{array}
\usepackage[caption=true,font=normalsize,labelfont=sf,textfont=sf]{subfig}
\usepackage{textcomp}
\usepackage{stfloats}
\usepackage{url}
\usepackage{verbatim}
\usepackage{graphicx}
\usepackage[noadjust]{cite}
\usepackage{hyperref}
\usepackage[utf8]{inputenc}
\usepackage{color}
\usepackage{stmaryrd}

\usepackage{mathtools}
\usepackage[margin=1in]{geometry}
\usepackage[sort&compress, numbers]{natbib}

\renewcommand{\epsilon}{\varepsilon}
\renewcommand{\phi}{\varphi}
\newcommand*{\N}{\mathbb{N}}

\newcommand*{\R}{\mathbb{R}}

\newcommand*{\C}{\mathbb{C}}

\newcommand{\norm}[1]{\left \| #1 \right \|}
\newcommand{\abs}[1]{\left|#1 \right|}
\newcommand{\skal}[1]{\left \langle #1 \right \rangle}

\DeclareMathOperator{\dom}{dom}

\makeatletter
\def\blfootnote{\xdef\@thefnmark{}\@footnotetext}
\makeatother

\theoremstyle{definition}
\newtheorem{defi}{Definition}[section]

\theoremstyle{plain}
\newtheorem{thm}[defi]{Theorem}
\newtheorem{lem}[defi]{Lemma}
\newtheorem{cor}[defi]{Corollary}
\newtheorem{prop}[defi]{Proposition}

\begin{document}
\title{Accelerated Griffin-Lim algorithm: \\ A fast and provably converging numerical method for phase retrieval}

\author{Rossen Nenov$^{ \dagger}$, Dang-Khoa Nguyen$^{ \star}$, Peter Balazs$^{ \dagger}$ and Radu Ioan Bo\c{t}$^{ \star}$ \\ $^{\dagger}$ Acoustics Research Institute, Vienna, Austria \\ $^{\star}$ University of Vienna, Austria}

%\author{Rossen Nenov$^{ \dagger}$ \qquad Dang-Khoa Nguyen$^{\star}$ \qquad Peter Balazs$^{\dagger}$}

%~\IEEEmembership{Staff,~IEEE,}
        % <-this % stops a space
%\blfootnote{This work was supported by the Austrian Science Fund FWF-project NoMASP (“Nonsmooth nonconvex optimization methods for acoustic signal processing”; P 34922-N). The authors would like to express their gratitude to Dr. Nicki Holighaus (Austria Academy of Sciences) for his valuable input and ideas in the numerical experiments section.}
%\blfootnote{Manuscript received --; revised --.}

%\address{$^{\dagger}$ Acoustics Research Institute, Vienna, Austria \\ $^{\star}$ University of Vienna, Austria} 
\markboth{Journal of \LaTeX\ Class Files,~Vol.~14, No.~8, August~2021}%
{Shell \MakeLowercase{\textit{et al.}}: A Sample Article Using IEEEtran.cls for IEEE Journals}


\maketitle

\begin{abstract}
    The recovery of a signal from the magnitudes of its transformation, like the Fourier transform, is known as the phase retrieval problem and is of big relevance in various fields of engineering and applied physics. In this paper, we present a fast inertial/momentum based algorithm for the phase retrieval problem and we prove a convergence guarantee for the new algorithm and for the Fast Griffin-Lim algorithm, whose convergence remained unproven in the past decade. In the final chapter, we compare the algorithm for the Short Time Fourier transform phase retrieval {with} the Griffin-Lim algorithm and FGLA and {to} other iterative algorithms typically used for this type of problem. 
\end{abstract}


%%%%%%%%%%%%%

\section{Introduction}\blfootnote{This work was supported by the Austrian Science Fund FWF-project NoMASP (“Nonsmooth nonconvex optimization methods for acoustic signal processing”; P 34922-N). The authors would like to express their gratitude to Dr. Nicki Holighaus (Austria Academy of Sciences) for his valuable input and ideas in the numerical experiments section.}
%\blfootnote{Manuscript received --; revised --.}

Reconstructing the phase of a signal from phaseless measurements of its (Short Time) Fourier transform is a pervasive challenge, called the phase retrieval problem. It arises in a great amount of areas of applications, most prominently in audio processing \cite{ac1, ac2}, imaging \cite{imag1,imag2}, and electromagnetic theory \cite{elec1,elec2}.  Therefore it is, to this day, an active research topic with numerous algorithms designed to find satisfactory solutions to the phase retrieval problem. These algorithms can be divided into two classes: non-iterative and iterative algorithms. For the \emph{Short Time Fourier transform} {(STFT)}, it was analysed in \cite{Nicki} that the performance of non-iterative algorithms strongly depends on the redundancy of the STFT. They perform better than iterative algorithms in terms of quality and speed for high redundancies, which are rarely considered in practice. For the lower, more common, redundancies the iterative methods gave more qualitative results than the non-iterative ones, which motivates their study.

The \emph{Griffin-Lim algorithm} (GLA) \cite{GLA} is a {well known} and widely used iterative algorithm based on the method of alternating projections and applied to the phase retrieval problem in the time-frequency setting. In optics, this algorithm is also known as the Gerchberg-Saxton algorithm \cite{GS}. Inspired by the \emph{Fast iterative shrinkage threshold algorithm} (FISTA \cite{FISTA}), {the authors of \cite{FGLA} proposed to introduce an} {inertial/momentum step} to GLA, which formed the \emph{Fast Griffin-Lim algorithm} (FGLA). {As a result, they obtained a method that converges faster than GLA} in practice and {recovers} {signals with lower reconstruction error}, but the convergence guarantee remained an open question. Since its introduction, FGLA gained considerable traction for phase retrieval in audio processing \cite{10.5555/2815664,9099086} and machine learning \cite{Marafiotia,SALEEM2022104389}, but still, a convergence result for this iterative algorithm was pending, which motivated this work. 

%In the process of studying FGLA and gaining an understanding of why and how it performs so well, we discovered a further adaptation to this method, which enhanced the performance. 
{In this work, we will present the \emph{Accelerated Griffin-Lim {algorithm}} (AGLA), a new iterative method, and prove a convergence result for it in a generalized setting covering a wide {range} of areas, where the phase retrieval problem arises, beyond the field of audio processing. This method is an extension of FGLA, and therefore our results cover FGLA as well.}

Furthermore, we will round out the theoretical results with the comparison of numerical simulations of AGLA against its predecessors and other iterative algorithms to {highlight} the improved numerical performance of our algorithm.

\section{The Phase Retrieval Problem}
A linear and injective transformation from $\C^L\to \C^M$ with $M\geq L$ can be written as a transformation matrix $T$ in $\C^{M \times L}$ with full column rank, for example, the discrete Fourier transform or the analysis operator of a finite frame \cite{FrameTh}. The vector $s\in [0,+\infty)^M$ will denote the measured magnitudes of the coefficients of the transform. The phase retrieval problem can be expressed mathematically as finding the signal $x^*\in \C^L$, whose transform coefficients match the magnitudes $s$, that is
$\abs{Tx^*}=s$, where $\abs{\cdot}$ is understood as the absolute value applied componentwise. {The feasibility, uniqueness, and stability of the phase retrieval problem have been studied in many works, most notably in \cite{Balan,Grohs}. The results in the present paper are {actually} independent of the solvability of the phase retrieval problem.}

In practice, measurements can include some noise and inaccuracies and therefore we are usually looking for solutions $x^*$ of the following minimization problem
\begin{align}
    \label{GOAL}
    \min_{x\in \C^L} \norm{\abs{Tx}-s},
\end{align}
where $\norm{\cdot}$ denotes the norm of $\C^M$. In \cite{FGLA} it was proposed to consider this problem, as the task in finding a vector $c^*$ in the set of coefficients admitted by the transformation matrix $T$, namely its range, which is as close as possible to the set of coefficients, whose magnitude match with $s$. The range of the transformation matrix $T$ will be denoted as 
\begin{align}
    C_1=\{c \in \C^M \mid \exists x \in \C^L: c=Tx\}.
\end{align}
This set is a linear subspace of $\C^L$. Let $C_2$ be the set of vectors whose magnitudes are equal to $s$, namely
 \begin{align}
              C_2=  \{c \in \C^{M} \mid \abs{c_i}=s_i \quad \forall i \in \{1,\dots,M\}\}.
\end{align}
The set $C_2$ is, by definition, compact. To formulate the problem \eqref{GOAL} as finding the closest point between two sets, we introduce two distance functions. The indicator function $\delta_C$ of a set $C$ is defined as
          \begin{align*}
              \delta_C({c})=\begin{cases} 0 & \text{if } {c} \in C, \\
              +\infty & \text{else,}
              \end{cases}
          \end{align*}
           and the distance function $d_C$ to a compact set $C$ is defined as
          \begin{align}\label{eq:dc}
              d_C({c})=\min_{{v} \in C } \norm{{c}-{v}} \quad \text{ for } c\in \C^M.
          \end{align}
%We aim to solve the nonconvex optimization problem
 Our aim is to solve the following optimization problem
\begin{align}
    \min_{c\in \C^M} f(c):= \delta_{C_1}(c)+\frac{1}{2}d_{C_2}^2(c), \label{obj}
\end{align}
which is nonconvex and has a close connection to the original problem \eqref{GOAL}, as can be seen later.
\color{black}
%We will denote $f(c):=\delta_{C_1}(c)+\frac{1}{2}d_{C_2}^2(c) $ as our objective function

For finding the closest point between two sets iteratively, projection operators are a common tool. Since $C_1$ is a linear subspace spanned by $T$, we can write its orthogonal projection as 
\begin{align*}
    P_{C_1}( {c})=TT^\dagger  {c},
\end{align*}
 where $T^\dagger$ denotes the pseudo-inverse of $T$, which is well-defined since $T$ is assumed to have full column rank \cite{FrameTh}.
%we will denote by  $\overline{P}_{C_2}:\C^L\rightrightarrows \C^L$ the possibly set valued projection 
%\begin{align*}
%    \overline{P}_{C_2}(x)=\argmin_{c\in C_2}\norm{x-c},
%\end{align*}
%which maps $x$ to the set of its closest points in $C_2$.

Since $C_2$ is a nonconvex closed set,  the projection of a vector $ {c}$ onto $C_2$ is nonempty but not always unique. Therefore we will denote by  $\overline{P}_{C_2}:\C^M\rightrightarrows \C^M$ the possibly set valued projection 
\begin{align*}
    \overline{P}_{C_2}( {c})=\arg\min_{{v} \in C_2}\norm{ {c}- {v}},
\end{align*}
which maps $c$ to the set of its closest points in $C_2$. As in \cite{FGLA} we will denote by $P_{C_2}:\C^M\to \C^M$ a closed form for a possible choice of the projection onto $C_2$, which we will use in the algorithms
\begin{align}\label{def:PC2}
            (P_{C_2}(c))_i=\begin{cases} \frac{s_ic_i}{\abs{c_i}} & \text{ if }c_i\neq 0, \\
              {s}_i & \text{ if }c_i=0. \end{cases}
\end{align}
This projection is equivalent to scaling the elements of $c$ to have the magnitudes of $s$, without changing the phase. The following proposition, whose proof can be found in the appendix, gives another convenient way to write $d_{C_2}$ and states that this choice of $P_{C_2}$ indeed minimizes the distance function and is almost everywhere the unique element of $\overline{P}_{C_2}$. 
\begin{prop}\label{d2}
For all $c\in \C^M$ the distance to $C_2$ can be written as \begin{align}
    d_{C_2}(c)=\norm{\abs{c}-s} \label{dc2}
\end{align} and $d_{C_2}(c)=\norm{c-P_{C_2}(c)}$ holds. If $c \notin D$, where \\
\centerline{$ D:=\{c\in \C^{M} \mid \exists i 
\in \{1,...,M\} : c_i=0 \text{ and } s_i\neq 0\}$,}
then $P_{C_2}(c)$ is the unique closest point to $c$ in $C_2$, namely $\overline{P}_{C_2}(c)=\{P_{C_2}(c)\}.$
\end{prop}
%The set $D$ is the set of points where the projection onto $C_2$ is not unique. 
{The set D consists of the points, where the magnitude of a coefficient is zero while respective measured one $s_{i}$ is not. This set has been observed to be problematic, for example, in the work \cite{Pole} it is proven that the Phase Derivative of the STFT is numerically unstable, i.e., there is a peculiar pole phenomenon at points in this set.}
It will have an important role in the convergence proof of the iterates, as we can only guarantee differentiability of $\frac{1}{2}d_{C_2}^2$ on $\C^M\setminus D$. 

Proposition \ref{d2} motivates the choice of the objective function $f$, since for $c\in C_1$ we see that $f(c)$ coincides with the function value of \eqref{GOAL} at $T^\dagger c$. One can show that for any local/global minimizer $c^*$ of \eqref{obj}, $x^*=T^\dagger c^*$ is a local/global minimizer of \eqref{GOAL} and vice versa.

\color{black}
%%%%%%%%%%%%%%%%%%%

\section{The Algorithms}

In 1984, Griffin and Lim presented the algorithm known as \emph{Griffin-Lim algorithm} (\hyperref[alg:GLA]{GLA}). They showed that the iterates of the algorithm converge to a set of critical points of a magnitude-only distance measure and that the objective function values of the iterates are non-increasing. {In the following, $N$ denotes the amount of iterations of the algorithms}, which can be chosen as $+ \infty$ as well.
{Here, without any risk of confusion, we use the subscript $n$ to mean the $n^{th}$ iteration, like $c_{n}, t_{n}$. On the other hand, the subscript $i$ means $c_i$ is the $i^{th}$ component of the vector $c \in \C^{M}$.}
\begin{algorithm}[H]
\caption{Griffin-Lim algorithm}\label{alg:GLA}
\begin{algorithmic}
\STATE 
\STATE {\textsc{Initialize }}$c_0\in \C^M$ \smallskip
\STATE \hspace{0.5cm}$ \textbf{Iterate for } n=1,\dots,N  $
\STATE \hspace{0.5cm}$ c_n =  P_{C_1}(P_{C_2}(c_{n-1})) $\smallskip 

\STATE {\textsc{Return }} $T^\dagger c_N$
\end{algorithmic}

\end{algorithm}

In 2013, Perraudin, Balazs and Søndergaard stated that \hyperref[alg:GLA]{GLA} can be seen as the method of alternating projections of the iterates onto $C_2$ and $C_1$ and proposed the \emph{Fast Griffin-Lim} (\hyperref[alg:FGLA]{FGLA}) by adding an inertial/momentum step {\cite{FGLA}. The algorithm is} based on the idea of the inertial proximal-gradient method, {in the spirit of the Heavy Ball method \cite{HBF} and Fast iterative shrinkage threshold algorithm (FISTA) \cite{FISTA},} {which also works in the nonconvex setting \cite{Bot-Csetnek-Laszlo}.}

{This algorithm achieved better results  {than} \hyperref[alg:GLA]{GLA} in numerical experiments, but a convergence guarantee was not addressed.}

\begin{algorithm}[H]
\caption{Fast Griffin-Lim algorithm }\label{alg:FGLA}
\begin{algorithmic}
\STATE 
\STATE {\textsc{Initialize }}$c_0\in \C^M$, $t_0\in C_1$ and $\alpha >0$ \smallskip
\STATE \hspace{0.5cm}$ \textbf{Iterate for } n=1,\dots,N  $
\STATE \hspace{0.5cm}$ t_n =  P_{C_1}(P_{C_2}(c_{n-1})) $
\STATE \hspace{0.5cm}$ c_n = t_n + \alpha (t_n-t_{n-1}), $  \smallskip 

\STATE {\textsc{Return }} $T^\dagger c_N$
\end{algorithmic}

\end{algorithm}

In this paper we present a further modification to \hyperref[alg:GLA]{GLA}, by adding a second inertial sequence $(d_n)_{n\in\N}$, which will not be projected. Its purpose is to stabilize the algorithm and avoid getting stuck at the points, in which \hyperref[alg:FGLA]{FGLA} stops at, {while the distance between the projection of $(c_n)_{n\in\N}$ and the nonprojected $(d_n)_{n\in\N}$ is still large}. A similar idea can be found in {\cite{Bot-Csetnek-Laszlo,Laszlo}}.

\begin{algorithm}[H]
\caption{Accelerated Griffin-Lim algorithm}\label{alg:AGLA}
\begin{algorithmic}
\STATE 
\STATE {\textsc{Initialize }}$c_0\in \C^L$, $t_0,d_0\in C_1$ and $\alpha, \beta, \gamma>0$ \smallskip
\STATE \hspace{0.5cm}$ \textbf{Iterate for } n=1,\dots,N  $
\STATE \hspace{0.5cm}$ t_n = (1-\gamma)d_{n-1}+\gamma P_{C_1}(P_{C_2}(c_{n-1})) $
\STATE \hspace{0.5cm}$ c_n = t_n + \alpha (t_n-t_{n-1}), $ 
\STATE \hspace{0.5cm}$  d_n = t_n + \beta (t_n-t_{n-1}) $ \smallskip 

\STATE {\textsc{Return }} $T^\dagger c_N$
\end{algorithmic}

\end{algorithm}
\hyperref[alg:AGLA]{AGLA} is a generalization of \hyperref[alg:FGLA]{FGLA} since for $\gamma=1$ the generated sequences by \hyperref[alg:FGLA]{FGLA} and \hyperref[alg:AGLA]{AGLA} coincide. In the following chapter, we will state parameter choices of $\alpha$, $\beta$, and $\gamma$, for which we prove the convergence of the algorithm.

%%%%%%%%%%%5

\section{Convergence of the function values}
In this chapter, we prove the convergence result for \hyperref[alg:AGLA]{AGLA}. We will use the following identity, which is a generalization of the parallelogram law, typically used in the proof of convergence of the function values of the iterates for algorithms with inertial sequences. 
{It will be used several times in the analysis. It can be shown by a simple calculation, hence we omit the detail. Precisely,}
for any two vectors $a,b\in \C^M$ and any two real numbers $\tau, \sigma \in \R$, it holds
    \begin{multline}  
	\left\lVert \tau a + \sigma b \right\rVert ^{2} = \left( \tau + \sigma \right) \tau \left\lVert a \right\rVert ^{2} + \left( \tau + \sigma \right) \sigma \left\lVert b \right\rVert ^{2}\\- \tau \sigma \left\lVert a - b \right\rVert ^{2}. \label{Star}
\end{multline}

%%%%%%%%%%%%%%%%%%

Using this identity we will state the proof of the function values of the iterates generated by \hyperref[alg:AGLA]{AGLA} for certain parameter regimes.
\begin{thm} \label{thm:Main}
Let $(c_n)_{n\in \N}$, $(d_n)_{n\in \N}$ and $(t_n)_{n\in \N}$ be the sequences generated by \hyperref[alg:AGLA]{AGLA}.
   Suppose that
   %$0<\gamma<2$, $0\leq 2\beta \abs{1-\gamma}<2-\gamma$ and
   \begin{equation}
   \label{cond1}
       0<\gamma<2 \textrm{ and } 0 \leq 2 \beta \abs{1-\gamma} < 2-\gamma ,
   \end{equation}
   and
   \begin{align}
       0 \leq \alpha < \begin{cases}
           \left( 1 - \frac{1}{\gamma} \right) \beta + \frac{1}{\gamma} - \frac{1}{2} & \text{ if } 0 < \gamma \leq 1, \\
           \frac{1}{ 2\beta(\gamma-1) + \gamma  } - \frac{1}{2} &  \text{ if } 1 < \gamma < 2 .
       \end{cases} 
       \label{cond}
   \end{align}
   
   %\begin{align}
    %  \gamma \in (0,1] \  \text{and} \  \begin{cases}\gamma(2\alpha-2\beta+1) < 2-2\beta \\ \gamma (\alpha^2-\alpha\beta+\beta)< \alpha -\alpha\beta+\beta 
     % \end{cases} \label{cond1}
   %\end{align}
   
    %or 
     %  \begin{align}
     % \gamma >1  \ \text{and} \ \begin{cases}\gamma(1+2\alpha)(1+2\beta) < 2+2\beta \\ \gamma (\alpha^2-3\alpha\beta-\beta)< \alpha -3\alpha\beta-\beta 
     % \end{cases} \label{cond2}
   %\end{align}
Then the following statements are true
\begin{enumerate}
    \item There exist constants $K_1>K_2>0$ such that the following descent property holds for all $n\geq 1$
    \begin{align*}
    d_{C_2}^2(c_n)+K_1\Delta_{t_n}^2  &\leq d_{C_2}^2(c_{n-1})+K_2\Delta_{t_{n-1}}^2,
\end{align*}
where $\Delta_{t_n}:=\norm{t_n-t_{n-1}}$. Therefore $\Delta_{t_n}$ converges to zero and $\lim_{n \to + \infty} d_{C_2}^2(c_n) \in \R$ exists.
\item[(ii)] $(c_n)_{n\in \N}$ is a bounded sequences and 
     every cluster point $c^*$ of $(c_n)_{n\in \N}$ fulfills
\begin{align*}
    P_{C_1}(P_{C_2}(c^*))=c^*.
\end{align*}
\end{enumerate} 
\end{thm}
\begin{proof}
Let $n\geq 1$. By the definition of the algorithm, the sequences $(c_n)_{n\geq 1}$, $(t_n)_{n\in N}$ and $(d_n)_{n\in\N}$ lie in the linear subspace $C_1$. This observation will be useful throughout the proof. 

$P_{C_1}$ is the orthogonal projection onto the linear subspace $C_1$ and therefore the identity
\begin{align}  \label{orth.}
    \norm{P_{C_1}(x)-x}^2+\norm{y}^2=\norm{P_{C_1}(x)-x+y}^2 
\end{align}
holds { for all $(x,y) \in \C^L \times C_1$}. Define $x_n:=P_{C_2}(c_{n-1})$ and $y_n=P_{C_1}(x_n)$, then for arbitrary $z\in C_1$ it holds $z-y_n\in C_1$. Hence according to \eqref{orth.} we have
\begin{equation}
\label{PR1}
\norm{y_n-x_n}^2+\norm{z-y_n}^2 = \norm{z-x_n}^2.
\end{equation}
By rewriting the definition of the algorithm, we see that
\begin{align} \label{PR1.5}
    \frac{1}{\gamma} t_n+\frac{\gamma-1}{\gamma}d_{n-1}=P_{C_1}(P_{C_2}(c_{n-1})),
\end{align}
and therefore we can write $y_n=\frac{1}{\gamma} t_n+\frac{\gamma-1}{\gamma}d_{n-1}$.
Since $z-y_n = \frac{1}{\gamma} (z-t_n)+\frac{\gamma-1}{\gamma}(z-d_{n-1})    $ the following expression can be expanded using the identity \eqref{Star}
\begin{multline}\label{PR2}
    \norm{z-y_n}^2  
    = \frac{1}{\gamma} \norm{z-t_n }^2 + \frac{\gamma-1}{\gamma} \norm{z-d_{n-1}  }^2 \\- \frac{\gamma-1}{\gamma^2}\norm{t_n-d_{n-1} }^2 .
\end{multline} 
 Combining \eqref{PR1} and \eqref{PR2} leads to
\begin{multline}
    \norm{y_n-x_n}^2-\frac{\gamma-1}{\gamma^{{2}}} \norm{t_n-d_{n-1}}^2 = \\ \norm{z-x_n}^2  - \frac{1}{\gamma} \norm{z-t_n}^2 - \frac{\gamma-1}{\gamma} \norm{z-d_{n-1} }^2.    \label{PR3} 
\end{multline}
The left hand side is independent of the choice of $z$, hence we can substitute $z$ by $c_n$ and $c_{n-1}$ and equate both right hand sides of \eqref{PR3} to get
\begin{align}
          \norm{c_{n}-x_n}^2 - \frac{1}{\gamma} \norm{c_{n}-t_n }^2 - \frac{\gamma-1}{\gamma} \norm{c_{n}-d_{n-1}  }^2 =  \notag\\
    d_{C_2}^2(c_{n-1}) - \frac{1}{\gamma} \norm{c_{n-1}-t_n }^2 - \frac{\gamma-1}{\gamma} \norm{c_{n-1}-d_{n-1}  }^2 ,\notag \\  \label{PR4} 
\end{align}
 where we used the fact that $d_{C_2}(c_{n-1})=\norm{c_{n-1}-x_n}.$ By the definition of the generated sequences, we can see that the following identities hold
\begin{align}
    \norm{c_n-t_n}^2&=\alpha^2\Delta_{t_n}^2\label{i ct}, \\
    \norm{c_{n}-d_{n}}^2&=(\alpha-\beta)^2\Delta_{t_n}^2, \label{i cd} \\
    \norm{t_{n}-d_{n}}^2&=\beta^2\Delta_{t_n}^2. \label{i td}
\end{align}
Furthermore, we can use the property that the distance of $c_n$ to $C_2$ is not larger than the distance of $c_n$ to an arbitrary point in $C_2$
\begin{align}
    d_{C_2}(c_n)=\min_{x\in C_2}\norm{c_{n}-x}\leq \norm{c_{n}-x_n}.\label{minD}
\end{align}
Applying \eqref{i ct}, \eqref{i cd} and \eqref{minD} into \eqref{PR4} yields
\begin{multline} 
    d_{C_2}^2(c_n)- \frac{\alpha^2}{\gamma}\Delta_{t_n}^2 - \frac{\gamma-1}{\gamma} \norm{c_{n}-d_{n-1}  }^2 \leq \\ d_{C_2}^2(c_{n-1}) - \frac{1}{\gamma} \norm{c_{n-1}-t_n }^2 - \frac{\gamma-1}{\gamma}(\alpha-\beta)^2\Delta_{t_{n-1}}^2. \label{PR5}
\end{multline}
At this point, we need to distinguish between the two cases $\gamma \in (0,1]$ and $\gamma > 1$. \smallskip \\
\textbf{Case 1:} Assume that $\gamma \in (0,1]$. By the definition of the sequences, we see that 
\begin{align*}
c_{n}-d_{n-1}&= (1+\alpha)(t_n-t_{n-1})-\beta(t_{n-1}-t_{n-2}) \\
    t_n-c_{n-1}&=  (t_n-t_{n-1})-\alpha(t_{n-1}-t_{n-2}) 
\end{align*}
and by applying the identity \eqref{Star} and leaving out a positive term, we get the following estimations
\begin{align*}
     \norm{c_{n}-d_{n-1}}^2&\geq (1+\alpha-\beta) ((1+\alpha)\Delta_{t_n}^2 - \beta \Delta_{t_{n-1}}^2 ), \notag \\
     \norm{t_n-c_{n-1}}^2 &\geq (1-\alpha)(\Delta_{t_n}^2-\alpha \Delta_{t_{n-1}}^2). %\label{PR6}
\end{align*}
Applying these estimates into \eqref{PR5} leads to
\begin{align} \label{PR7}
    d_{C_2}^2(c_n)+K_1\Delta_{t_n}^2 \leq  d_{C_2}^2(c_{n-1})+K_2\Delta_{t_{n-1}}^2,
\end{align}
with $K_1:=\frac{1-\gamma}{\gamma}(1+2\alpha+\alpha^2-\beta-\alpha\beta)+\frac{1}{\gamma}(1-\alpha-\alpha^2)$ and $K_2:=\frac{1-\gamma}{\gamma}(\beta-\alpha \beta + \alpha^2)+\frac{1}{\gamma}(\alpha-\alpha^2)$.
{Lemma \ref{lem:para} ensures that} \eqref{cond1} and \eqref{cond} imply $K_1 > K_2 > 0.$ \\
%%%%%%%%%%%%%%%%%%%%%%%%%%%%%%%%%%%%
\smallskip
\textbf{Case 2:} Assume that $\gamma > 1$. Similarly, we see that 
\begin{align*}
c_{n}-d_{n-1}&= (1+\alpha)(t_n-t_{n-1})+\beta(t_{n-2}-t_{n-1}) \\
    t_n-c_{n-1}&=  (t_n-t_{n-1})-\alpha(t_{n-1}-t_{n-2}) 
\end{align*}
and by applying the identity \eqref{Star}, we get 
\begin{align*}
     \norm{c_{n}-d_{n-1}}^2&\leq (1+\alpha+\beta) ((1+\alpha)\Delta_{t_n}^2 + \beta \Delta_{t_{n-1}}^2 ), \notag \\
     \norm{t_n-c_{n-1}}^2 &\geq (1-\alpha)(\Delta_{t_n}^2-\alpha \Delta_{t_{n-1}}^2). %\label{PR6.5}
\end{align*}
Applying these estimates into \eqref{PR5} leads to
\begin{align} \label{PR7.5}
    d_{C_2}^2(c_n)+K_1\Delta_{t_n}^2 \leq  d_{C_2}^2(c_{n-1})+K_2\Delta_{t_{n-1}}^2,
\end{align}
with $K_1:=\frac{1-\gamma}{\gamma}(1+2\alpha+\alpha^2+\beta+\alpha\beta)+\frac{1}{\gamma}(1-\alpha-\alpha^2)$ and $K_2:=\frac{1-\gamma}{\gamma}(\alpha^2-\beta-3\alpha \beta)+\frac{1}{\gamma}(\alpha-\alpha^2)$. {Lemma \ref{lem:para}} asserts that \eqref{cond1} and \eqref{cond} imply $K_1>K_2>0.$

On the one hand, we can deduce from \eqref{PR7} and \eqref{PR7.5} that the sequence
\begin{align}
    \left( d_{C_2}^2(c_n)+K_2\Delta_{t_n}^2 \right) _{n\in \N} \label{PR7.7}
\end{align}
is positive and non-increasing as $n$ increases, therefore the sequence \eqref{PR7.7} converges as $n \to + \infty$. On the other hand, we can rewrite \eqref{PR7} and \eqref{PR7.5} to get
\begin{multline*}
        0\leq (K_1-K_2)\Delta_{t_n}^2 \leq  d_{C_2}^2(c_{n-1}) -  d_{C_2}^2(c_n) \\+ K_2(\Delta_{t_{n-1}}^2-\Delta_{t_n}^2).
\end{multline*}
If we sum up this inequality we deduce that
\begin{multline*}
   (K_1-K_2) \sum_{j=2}^N\Delta_{t_j}^2 \leq d_{C_2}^2(c_{1}) -  d_{C_2}^2(c_N) \\+ K_2(\Delta_{t_{1}}^2-\Delta_{t_N}^2) .
\end{multline*}
Since {for all $n \geq 2$ it holds $ d_{C_2}^2(c_{0}) + K_2 \Delta_{t_{0}}^2 \geq 0$, by taking the limit $n \to + \infty$ we see that}
\begin{align*}
    \sum_{j=2}^{+\infty}\Delta_{t_j}^2 \leq d_{C_2}^2(c_{1}) + K_2 \Delta_{t_{1}}^2 < + \infty
\end{align*}
holds, and therefore $\Delta_{t_n} \to 0$ as $n \to + \infty.$ 
{Moreover, we already showed that \eqref{PR7.7} is converges}
and as a consequence $d_{C_2}(c_n)$ also converges as $n \to + \infty$. Hence the sequence $(c_n)_{n\in \N}$ must be bounded, since its distance to the bounded set $C_2$ converges. By \eqref{i ct} and \eqref{i cd}, we see that $(\norm{c_n-t_n})_{n\in \N}$ and $(\norm{c_n-d_n})_{n\in \N}$ converge to zero as well. Using this observation and \eqref{PR1.5} we conclude 
\begin{align}\label{ConvCrit}
    P_{C_1}(P_{C_2}(c_n))-c_n \to 0 \text{ as } n \to + \infty.
\end{align}
Furthermore, since $c_n$ is a bounded sequence, there exist cluster points. For each cluster point $c^*$ 
\begin{align*}
    P_{C_1}(P_{C_2}(c^*))=c^*
\end{align*}
has to hold by \eqref{ConvCrit}. 
\end{proof}

In numerical experiments we will see that {whenever} $\beta$ {or} $\gamma$ are not chosen to fulfill \eqref{cond1}, then \hyperref[alg:AGLA]{AGLA} does not converge. This suggests that it might not be possible to extend the condition \eqref{cond1}. On the other hand, \eqref{cond1} and \eqref{cond} together are sufficient conditions to guarantee convergence. 
%This will be further analysed in Chapter VI.} 

%%%%%%%%%%%%%%%%

%%%%%%%%%%%%%%%%%%%%

\section{Convergence of the iterates}
In this section, we will associate $c\in \C^M$ with {$y\in \R^{M \times 2}$}, where $y_i=(\text{Re}(c_i),\text{Im}(c_i))$ for all $i\in \{1,\dots, M\}$, to apply the subdifferential calculus and the \emph{Kurdyka–\L{}ojasiewicz-property} (KL-property), which are defined for functions from $\R^{M \times 2}$ to ${\overline{\R} := \R \cup \{ \pm \infty \}}$. The set $C_1\subseteq \R^{M \times 2}$ remains a linear subspace and the set $C_2$ can be written as
\begin{align}
    C_2=\{ y\in \R^{M \times 2}: \norm{y_i}=s_i, \forall i=1,...,M \}. \label{C2alg}
\end{align}
The domain of an extended real-valued function $f$ is defined as $\dom f = \{ {y}:f( {y})<+\infty\}. $ For our objective function \eqref{obj} it holds $\dom f = C_1$. A function is called proper if it never attains the value $-\infty$ and its domain is a nonempty set. Since our objective function is nonconvex and not everywhere differentiable, we introduce the notion of the limiting subdifferential following  \cite[Definition 8.3]{rockafellar}.

\begin{defi} For the proper function $f:\R^{M \times 2}\to \overline{\R}$
and point $\Bar{y}\in \dom f$, the \textit{regular subgradient} $\hat{\partial} f(\Bar{y})$ is defined as the set of vectors $\Bar{v} \in \R^{M \times 2}$ which fulfill
\begin{align*}
   \liminf_{y\to \Bar{y}}\frac{f(y)-f(\Bar{y})-\skal{\Bar{v},y-\Bar{y}}}{\norm{y-\Bar{y}}}\geq 0,
\end{align*}
and the \textit{limiting subdifferential} $\partial f(\Bar{y})$ is defined as the set of vectors $\Bar{v}\in \R^{M \times 2}$ such that there exist sequences $(y_n)_{n\in \N}\subseteq \R^{M \times 2}$ and $v_n \in \hat{\partial} f(y_n)$ such that $y_n \to \Bar{y}$, $f(y_n)\to f(\Bar{y})$ and $v_n \to \Bar{v}$ as $n \to + \infty$.

\end{defi}
%\tcr{The subdifferential is a set-valued operator in general. When $f$ is continuously differentiable on a neighborhood of $\Bar{y}$, then $\partial f(\Bar{y})$ is a singleton and $\nabla f(\Bar{y})$ is its unique element \cite[Corrolary 9.19]{rockafellar}.}

{The subdifferential is a set-valued operator and its domain is defined as $ \dom \partial F :=\{u\in \R^{M \times 2}: \partial F(u)\neq \emptyset \}$. If $f$ is continuously differentiable on a open set, then $\partial f$  is single-valued and thus coincide with $\nabla f$ on this set \cite[Corrolary 9.19]{rockafellar}.}
%The subdifferential can is a set-valued operator, which is unique at the points, where $f$ is continuously differentiable. 

The following proposition states the formula for the subgradient of our objective function,  {whose proof is postponed to the appendix}.
\begin{prop} \label{prop:Subdiff}
For the function $f:\R^{M \times 2}\to \overline{\R}$, $y\mapsto \delta_{C_1}(y)+\frac{1}{2}d_{C_2}^2(y)$ the limiting subdifferential $\partial f: \R^{M \times 2} \rightrightarrows \R^{M \times 2}$ is given by
\begin{align*}
    \partial f(y) = C_1^\perp+y-{P}_{C_2}(y), \text{ for } y\in C_1 \setminus D,
\end{align*}
where 

\centerline{$D=\{y\in \R^{M \times 2} \mid \exists i 
\in \{1,...,M\} : y_i=0 \text{ and } s_i\neq 0\}$.} 
Furthermore for $y\in C_1 \setminus D$ it holds
\begin{align*}
    d_{\partial f(y)}(0)\leq \norm{y-P_{C_1}(P_{C_2}(y)) } .
\end{align*}
\end{prop}
%To write the subgradient for vectors $x$ in the set of nonregularity $D$ apriori knowledge about $C_1$ and $C_2$ is required. 
Here for simplicity, we keep the notation $f$ for the objective function and write $f( {y})$ when it maps from $\R^{M \times 2}$ to $\overline{\R}$, and $f(c)$ when $f \colon \C^{M}\to \overline{\R}$.
%Using the notion of the limiting subdifferential we deduce that the cluster points of the algortihm can be local minima of the objective function.

%\begin{prop} \label{optimality} Let $(c_n)_{n\in \N}$ be a sequence generated by \hyperref[alg:AGLA]{AGLA} under the same assumptions as in Theorem \ref{thm:Main}. Then any cluster point $c^*\notin D$ of the sequence is a local minimizer of $f(y)=\delta_{C_1}(y)+\frac{1}{2}d_{C_2}^2(y)$.
%\end{prop}
Before we introduce the KL-property, we have to define the class of concave and continuous functions.
\begin{defi}
Let $\eta \in (0,+\infty]$. We denote by $\Phi_\eta$ all function $\phi: [0,\eta)\to [0,+\infty)$ which satisfy the following conditions
\begin{enumerate}
    \item $\phi(0)=0$
    \item $\phi$ is $C^1$ on $(0,\eta)$ and continuous at $0$
    \item for all $s\in (0,\eta): \quad \phi'(s)>0$
\end{enumerate}
\end{defi}
We state the definition of the KL-property, which can be used to prove convergence iterates of first-order and second-order methods, with nonconvex objective functions {\cite{Attouch-Bolte}}. { Intuitively speaking, functions that satisfy this property are not too flat at their respective local minimizers and critical points.}
For $b>a$, we will write $\llbracket a<F<b \rrbracket$ to denote the level set $\{ u \in \R^{M \times 2}: a<F(u)<b \}$ of a function $F:\R^{M \times 2}\to  \overline{\R}$.

\begin{defi}Let $F:\R^{M \times 2}\to \overline{\R}$ be proper and lower semicontinuous. $\partial F$ denotes the subdifferential of $F$. The function $F$ is said to have the \emph{Kurdyka–\L{}ojasiewicz-property} at $\overline{u}\in \dom \partial F $ if there exist $\eta \in (0,+\infty]$, a neighborhood $U$ of $\overline{u}$ and a function $\phi \in \Phi_\eta$, such that for all 
\begin{align*}
    u\in U \cap \llbracket F(\overline{u})<F<F(\overline{u})+\eta \rrbracket,
\end{align*}
the following holds
\begin{align*}
    \phi'(F(u)-F(\overline{u})) d_{\partial F(u)}(0)\geq 1.
\end{align*}
The function $\phi$ is called the desingularizing function.
\end{defi}

The following result, taken from \cite[Lemma 6]{Bolte-Sabach-Teboulle}, {provides a uniformized KL property on a neighborhood and } will be crucial in our convergence analysis.
\begin{lem}
	\label{lem:uniformized}
	Let $\Omega$ be a compact set and $F \colon \R^{M \times 2}\to \overline{\R}$ be a proper and lower semicontinuous function. Assume that $F$ is constant on $\Omega$ and satisfies the KL property at each point of $\Omega$. Then there exist $\varepsilon > 0, \eta > 0$ 
	and $\varphi \in \Phi_{\eta}$ such that for every $\overline{u} \in \Omega$ and  every element $u$ in the intersection
	\begin{equation*}
	\left\lbrace u \in \R^{M \times 2} \colon d_{\Omega}(u)< \varepsilon \right\rbrace \cap \llbracket F \left( \overline{u} \right) < F  < F \left( \overline{u} \right) + \eta \rrbracket
	\end{equation*}
	the following holds:
	\begin{equation*}
	\varphi ' \left( F \left( u \right) - F \left( \overline{u} \right) \right) d_{\partial F(u)}(0) \geq 1 .
	\end{equation*}
\end{lem}

The KL-property holds for a broad class of functions, especially for the indicator functions and distance functions of semi-algebraic sets as stated in \cite{SemiAlg}. Since $C_1$ as a linear subspace is algebraic and $C_2$ is algebraic as we see in \eqref{C2alg}, we know that our objective function $f$ in \eqref{obj} has the KL-property.

%%%%%%%%%%%%%%%%%%%%%%%%

For the proof of the convergence of the iterates, we introduce the regularized version of $f$, namely $F_K: \R^{M \times 2} \times \R^{M \times 2}$  defined as
\begin{align*}
    F_K(y,z)= \delta_{C_1}(y)+\frac{1}{2}d^2_{C_2}(y) + \frac{K}{2\alpha^2} \norm{y-z}^2,
\end{align*}
with $K>0$. Then by Proposition 10.5 and Corollary 10.9 of \cite{rockafellar} the formula of the subdifferential $\partial F_K: \R^{M \times 2}\times \R^{M \times 2} \rightrightarrows \R^{M \times 2}\times \R^{M \times 2}$ is given as
\begin{align} \label{eq:subdifF}
    \partial F_K(y,z) =  \left\lbrace \partial f(y) + \frac{K}{\alpha^2}(y-z)\right\rbrace \times \left\lbrace \frac{K}{\alpha^2}(z-y) \right\rbrace.
\end{align}
Furthermore, by \cite{LiPong} we know that $F_K$ inherits the KL-property from $f$. Simple calculations show that $F$ is non-increasing, and the distance of the subgradient to 0 can be estimated from above.

\begin{prop} Let $(c_n)_{n\in \N}$, $(d_n)_{n\in \N}$ and $(t_n)_{n\in \N}$ be the sequence generated by \hyperref[alg:AGLA]{AGLA} and assume that \eqref{cond1} and \eqref{cond} hold. Then the following statements are true:
\begin{enumerate}
    \item There exist a constant $\kappa_1>0$ such that for all $n\in\N$ it holds
\begin{align}
    F_{K_2}(c_n,t_n)+\kappa_1\Delta_{t_n}^2 & \leq F_{K_2}(c_{n-1},t_{n-1}), \label{c1}
\end{align} 
where $K_2$ is defined as in Theorem \ref{thm:Main}.
    \item If $c_n \in D$ {for at most finitely many $n\in \N$}, where $D$ is defined as in Proposition \ref{prop:Subdiff}, then there exist {an integer} $m\in \N$ and a constant $\kappa_2>0$ such that for all $n\geq m$ it holds
\begin{align}
    d_{\partial F_{K_2}(c_n,t_n)}(0) & \leq \kappa_2(\Delta_{t_{n+1}}+\Delta_{t_n}).\label{c2}
\end{align}
\end{enumerate}
\label{prop:Freg}
\end{prop}
The proof of this proposition can be found in the appendix. 
Now we can state the proof that the generated iterates of \hyperref[alg:AGLA]{AGLA} converge by using the KL-property and the previous observations.

\begin{thm} \label{thm:convIT}
Let $(c_n)_{n\in \N}$, $(d_n)_{n\in \N}$ and $(t_n)_{n\in \N}$ be the sequence generated by \hyperref[alg:AGLA]{AGLA}. Assume that \eqref{cond1} and \eqref{cond} hold
%\begin{enumerate}
%    \item either $(c_n)_{n\in \N}$ becomes constant after finite times;
%    \item or $(c_n)_{n\in \N}$ converges to a local minimum of $f$ if $c_n \in D$  {for at most finitely many $n\in \N$}, where $D$ is defined as in Proposition \ref{prop:Subdiff}.
%\end{enumerate}
and that $c_n \in D$  {for at most finitely many $n\in \N$}, where $D$ is defined as in Proposition \ref{prop:Subdiff}. Then $(c_n)_{n\in \N}$ converges and the limit $c^*$ is a critical point of $f$, if $c^*\notin D$. 
\end{thm}

\begin{proof}
For simplicity we will denote $F_n=F_{K_2}(c_n,t_n)$ and $\lim_{n\to \infty}F_n=F^*$, which exists by Theorem \ref{thm:Main}. For any cluster point $c^*$ of the sequence $(c_n)_{n\in\N}$, we see that $F(c^*)=F^*$ by Theorem \ref{thm:Main}.
%By Proposition \ref{prop:Freg}, we know that there exists $\kappa_1,\kappa_2>0$ such that for all $n\geq m$
%\begin{align}
%    F_n+\kappa_1\Delta_{t_n}^2 &\leq F_{n-1}, \label{c1}\\
%    d_{\partial F_n}(0) &\leq \kappa_2 \left(\Delta_{t_{n+1}}  + \Delta_{t_n} \right).\label{c2}
%\end{align}
If there exists an integer $n_0 \in \N$ such that $F_{n_0}=F^*$, then the inequality \eqref{c1} implies that $\Delta_{t_n}=0$ for all $n \geq n_0$. Hence by the definition of the algorithm and \eqref{i ct}-\eqref{i td}, the sequences $(t_n)_{n\geq n_0}$, $(d_n)_{n\geq n_0}$ and $(c_n)_{n\geq n_0}$ are constant and the statement is proven.

Now assume that $F_{n}>F^*$ for all $n \in \N$. 
If $c_n \in D$ for at most finitely many $n\in \N$, then we can choose $m\in \N$ such that for all $n\geq m$ the vector $c_n\notin D$.
%By Proposition \ref{prop:Freg}, we know that there exists $\kappa_2>0$ such that for all $n\geq m$
%\begin{equation}
%    d_{\partial F_n}(0) \leq \kappa_2 \left(\Delta_{t_{n+1}}  + \Delta_{t_n} \right).\label{c2}
%\end{equation}
\color{black}

%If there exists a $n_0\geq m$ such that $(c_{n_0},t_{n_0})$ is a critical point of $F_{K_2}$, i.e.
%$0\in \partial F_k(c_{n_0},t_{n_0})$, then by \eqref{eq:subdifF} and Proposition \ref{prop:Subdiff} $t_{n_0}=c_{n_0}$ and $c_{n_0}=P_{C_1}(P_{C_2}(c_{n_0}))$, hence by the definition of the algorithm and \eqref{i ct}-\eqref{i td} the sequences $(t_n)_{n\geq n_0}$, $(d_n)_{n\geq n_0}$ and $(c_n)_{n\geq n_0}$ are constant sequences and the statement is proven.

%Now assume that the generated sequences do not reach a critical point after finitely many steps, namely that $d_{\partial F_n}(0)>0$ and $\Delta_{t_n}>0$ for all $n\geq m$. 
From Theorem \ref{thm:Main}, we know that the sequence $((c_n,t_n))_{n\in\N}$ is bounded, and therefore there exist cluster points.
Let us denote by $\Omega$ the set of all cluster points of the sequence $((c_n,t_n))_{n\in\N}$.
We can see that $\Omega$ is closed and also bounded. Moreover, the value of $F$ over $\Omega$ always equals $F^*$.
Since $F$ has the KL-property, according to Lemma \ref{lem:uniformized}, there exist $\varepsilon, \eta > 0$ and a function $\phi \in \Phi_\eta$ such that all element $(y,z)$ in the intersection
\begin{multline}
\left\lbrace (y,z) \in \R^{M \times 2} \times \R^{M \times 2} \colon d_{\Omega}((y,z))< \varepsilon \right\rbrace \\
\cap \llbracket F^{*} < F_{K_{2}} < F^{*} + \eta \rrbracket \label{intersection-all}    
\end{multline}
it holds
\begin{equation}
\label{KL-inequality}
\varphi ' \left( F_{K_{2}} \left( (y,z) \right) - F^{*} \right) d_{\partial F_{K_{2}} \left( (y,z) \right)}(0) \geq 1 .
\end{equation}

Since $F_n$ converges to $F^*$ and $d_{\partial F_{K_2}(c_n,t_n)}(0) \to 0$ as $n \to + \infty$ due to Proposition \ref{prop:Freg}, there exists $n_1\geq m$ such that $(c_n,t_n)$ belongs to the intersection \eqref{intersection-all} for all $n \geq n_1$.
This means by \eqref{KL-inequality} and \eqref{c2}, the inequality
\begin{multline}
    \kappa_2 \phi'(F_n-F^*) (\Delta_{t_{n+1}}+\Delta_{t_n}) \\
    \geq \phi'(F_n-F^*)d_{\partial F_n}(0) \geq 1 , \label{c3}
\end{multline}
    holds for all $n \geq n_1.$  
\color{black}    
    Since $\phi$ is concave and differentiable, we know that 
    \begin{align*}
        \phi(z)-\phi(y)\geq \phi'(z)(z-y).
    \end{align*}
    By choosing $z=F_n-F^*$ and $y=F_{n+1}-F^*$ then plugging \eqref{c1} and \eqref{c3}, we obtain after some rearranging
 \begin{multline*}
     \phi(F_n-F^*)-\phi(F_{n+1}-F^*) \geq  \frac{\kappa_1}{\kappa_2} \frac{\Delta^2_{t_{n+1}}}{\Delta_{t_{n+1}}+ \Delta_{t_{n}}}.
 \end{multline*}
 holds for $n\geq n_1$. By applying Lemma \ref{lem:ABC} in the appendix we can see that this implies that
 \begin{align*}
     \frac{9\kappa_2}{4\kappa_1}\left(\phi(F_n-F^*)-\phi(F_{n+1}-F^*) \right)\geq   2\Delta_{t_{n+1}}-\Delta_{t_{n}}
 \end{align*}
 holds. Summing up this inequality for  $j=n_1,\dots,\Bar{n}$ leads to
  \begin{multline*}
      \sum_{j=n_1}^{\Bar{n}}\Delta_{t_{j}+1}   \leq \frac{9\kappa_2}{4\kappa_1}\left(\phi(F_{n_1}-F^*)-\phi(F_{\Bar{n}+1}-F^*)\right)  \\
      + { \Delta_{t_{n_1}}}-\Delta_{t_{\Bar{n}+1}}.
 \end{multline*}
Since {$\phi$ is nonnegative, by passing $\Bar{n}\to \infty$, we deduce}
  \begin{align*}
      \sum_{j=n_1}^{+\infty}\Delta_{t_j} \leq   \frac{9\kappa_2}{4\kappa_1}\phi(F_{n_1}-F^*)+ { \Delta_{t_{n_1}}}< +\infty.
 \end{align*}
 This implies that $(t_n)_{n\in \N}$ is a Cauchy-sequence and therefore converges. Since $(c_n)_{n\in \N}$ and $(d_n)_{n\in \N}$ are linear combinations of $(t_n)_{n\in \N}$, we conclude that these series converge as well and the limit of $(c_n)_{n\in \N}$ {has to be $c^*$ according to {Theorem \ref{thm:Main}}.} Furthermore $0\in \partial f(c^*)$ by Proposition \ref{prop:Subdiff} if $c^*\notin D.$
\end{proof}

Since \hyperref[alg:AGLA]{AGLA} with $\gamma=1$ reduces to \hyperref[alg:FGLA]{FGLA}, we can state the following convergence properties of \hyperref[alg:FGLA]{FGLA} {to complement the result in \cite{FGLA}. It is a direct consequence of Theorem \ref{thm:Main} and \ref{thm:convIT}. 
\begin{cor}
   Let $(c_n)_{n\in \N}$ be the sequences generated by \hyperref[alg:FGLA]{FGLA}. If $\alpha \in (0,0.5)$, then
\begin{enumerate}
    \item {There exist constants $K_1>K_2>0$ such that the following descent property holds for all $n\geq 1$
    \begin{align*}
    d_{C_2}^2(c_n)+K_1\Delta_{t_n}^2  &\leq d_{C_2}^2(c_{n-1})+K_2\Delta_{t_{n-1}}^2,
\end{align*}
where $\Delta_{t_n}:=\norm{t_n-t_{n-1}}$. Therefore $\Delta_{t_n}$ converges to zero and $\lim_{n \to + \infty} d_{C_2}^2(c_n) \in \R$ exists. 
\item Furthermore $(c_n)_{n\in \N}$ is a bounded sequence and every convergent subsequence converges to a point $c^*$, which fulfills}
\begin{align*}
    P_{C_1}(P_{C_2}(c^*))=c^*
\end{align*}
and is therefore a critical point if $c^*\notin D$, where $D$ is defined as in Proposition \ref{prop:Subdiff}.
\item If $c_n \in D$ {for at most finitely many $n\in \N$}, then $(c_n)_{n\in\N}$ is a convergent sequence. 
\end{enumerate} 
\end{cor}
To summarize our results, we have established conditions and parameter regimes for \hyperref[alg:AGLA]{AGLA} and \hyperref[alg:FGLA]{FGLA} that guarantee that the generated sequences $(c_n)_{n\in \N}$ assert a decreasing property and that every cluster point $c^*\notin D$ is a local minimizer of our objective function. To guarantee convergence, we needed the condition $c_n \in D$ for at most finitely many $n\in \N$, since outside of $D$ the projection $\overline{P}_{C_2}$ is unique and $\frac{1}{2}d_{C_2}^2$ is continuously differentiable. %It is important to mention that these inequalities do not cover all possible choices of parameters for which we can observe convergence in the simulations. For example, it was observed in \cite{\hyperref[alg:FGLA]{FGLA}} that \hyperref[alg:FGLA]{FGLA} converges best for $\alpha=0.99$ even though these proofs only guarantee convergence for $\alpha \in (0,0.5)$. 
It is important to mention that the inequalities used in the statements above do not cover all possible choices of parameters for which we have observed convergence in the simulations. 
{A similar situation occurs even if $\frac{1}{2}d_{C_2}^2$ is assumed to have Lipschitz-continuous gradient, see \cite{Bot-Csetnek-Laszlo,Laszlo}. Notice that Lipschitz-continuity does not hold in our model, according to Proposition \ref{prop:Subdiff}.}
%\tcm{To be able to extend the parameter regimes in our forward approach, we believe that we need $\frac{1}{2}d_{C_2}^2$ to have a Lipschitz-continuous gradient, as this a usual requirement for similar methods \cite{FISTA}. This is not the case with our model, according to Proposition \ref{prop:Subdiff}.}
%\tcm{Moreover, the convergence analysis for the minimization method in this paper usually requires $\frac{1}{2}d_{C_2}^2$ to have a Lipschitz-continuous gradient. This is not the case with our model, according to Proposition \ref{prop:Subdiff}.} 

Most importantly, our proofs show that for \hyperref[alg:AGLA]{AGLA} and \hyperref[alg:FGLA]{FGLA} a good performance and minimizing properties for the phase retrieval problem can be guaranteed, when the parameters are well chosen.
%{The minimization method applied to the objective functions would require for $\frac{1}{2}d_{C_2}^2$ {to have} Lipschitz-continuous gradient {in the convergence analysis}} in a forward approach. }
%for a wider range of parameters.

%{Most importantly, the proofs show that for \hyperref[alg:AGLA]{AGLA} and \hyperref[alg:FGLA]{FGLA} a good performance and minimizing properties for the phase retrieval problem can be guaranteed, when the parameters are well chosen. }

\begin{figure}[h!]
    \centering
    \includegraphics[width=0.95\linewidth]{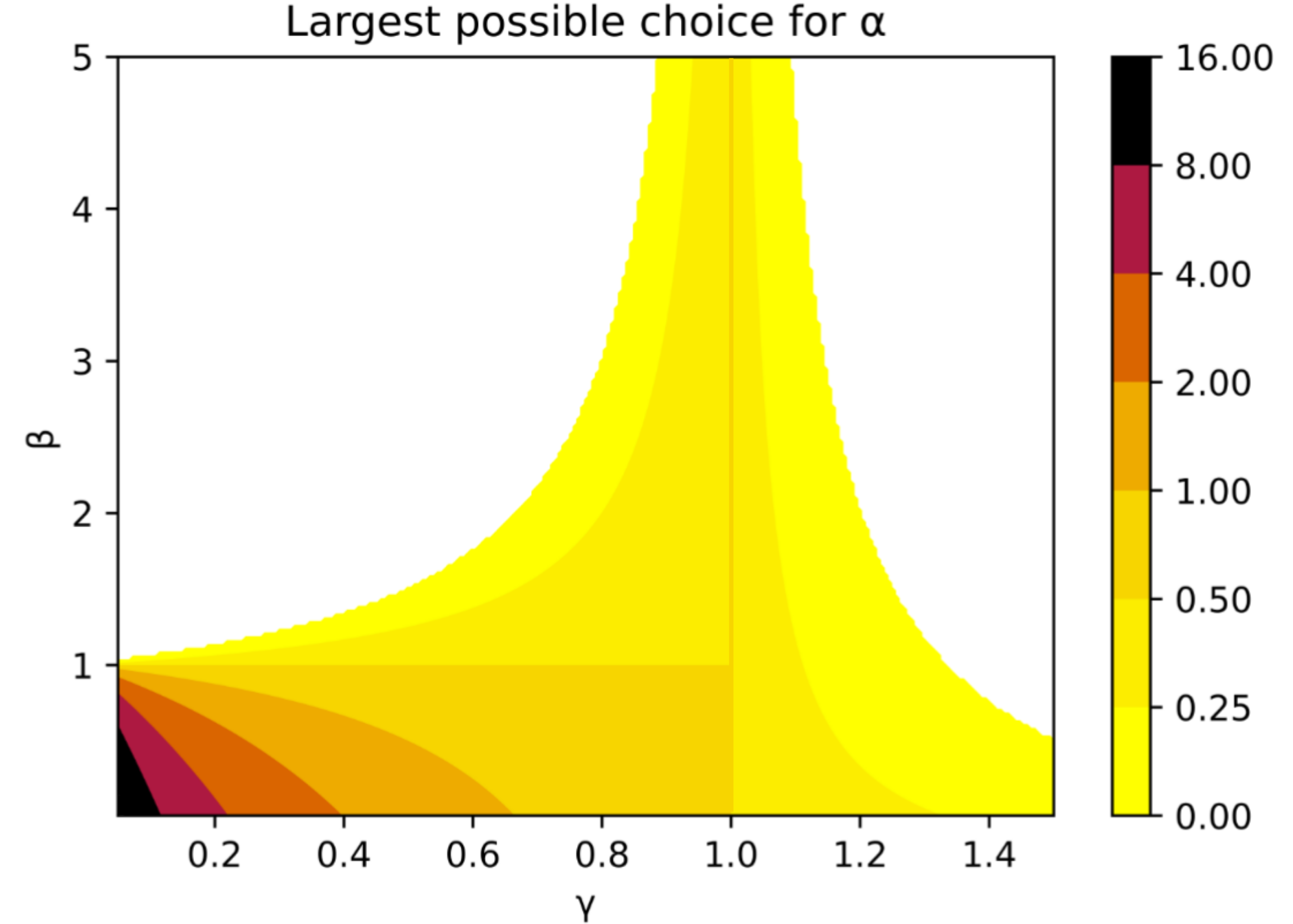}
    \caption{Largest possible choice for $\alpha$ in \hyperref[alg:AGLA]{AGLA} for given $\beta$ and $\gamma$ based on the conditions \eqref{cond1} and \eqref{cond} that guarantee convergence}
    \label{fig:parameter}
\end{figure}

In Figure \ref{fig:parameter} we plotted for fixed $\beta\in [0,5]$ and $\gamma \in [0.1,1.5]$ the upper bound for $\alpha$ s.t. $\alpha,\beta,\gamma $ satisfy the convergence guarantees \eqref{cond1} and \eqref{cond} of Theorem \ref{thm:Main}. The white areas symbolize the combinations of $\beta$ and $\gamma$, for which \eqref{cond1} is not fulfilled, i.e. where we expect \hyperref[alg:AGLA]{AGLA} not to converge. As noticed before, for $\gamma=1$ \hyperref[alg:AGLA]{AGLA} coincides with \hyperref[alg:FGLA]{FGLA} and \eqref{cond} reduces to $\alpha<0.5$ independently of $\beta$, which we can see in the Figure. 

Before we take a look at the numerical simulations we will state a corollary, which motivates for which iterates we are going to plot the function values.
\begin{cor} \label{thm:plot}
    Let $(c_n)_{n\in \N}$ be the sequence generated by \hyperref[alg:AGLA]{AGLA}. Assume that \eqref{cond1} and \eqref{cond} hold {and that $c_n \in D$  {for at most finitely many $n\in \N$}, where $D$ is defined as in Proposition \ref{prop:Subdiff}}. Then for $y_n=P_{C_1}(P_{C_2}(c_n))$ the following properties hold:
    \begin{enumerate}
        
            \item For $n\geq 0$ \begin{align*}
        d_{C_2}^2(y_n)\leq d_{C_2}^2(c_{n})-\norm{y_n-c_{n}}^2. 
    \end{align*}
        \item The sequence $(y_n)_{n\in \N}$  converges as $n \to + \infty$.

    \end{enumerate}

    \begin{proof}
       Let $n\geq 0$. By \eqref{orth.} we know that 
        \begin{align*}
            \norm{y_n-P_{C_2}(c_n) }^2=\norm{c_n-P_{C_2}(c_n)}^2-\norm{y_n-c_n}^2.
        \end{align*}
    By the definition of the distance function we see that $d_{C_2}^2(y_n)\leq \norm{y_n-P_{C_2}(c_n) }^2$ and $d_{C_2}^2(c_n)=\norm{c_n-P_{C_2}(c_n)}^2$ hold and by this (i) is proven. Furthermore by the results of the Theorems \ref{thm:Main} and \ref{thm:convIT} we can deduce by \eqref{PR1.5} that $(y_n)_{n\in \N}$ converges as well.
    \end{proof}
\end{cor}

%%%%%%%%%%%%%%%%%%%%%

\section{Numerical Experiments}
In this section we will present the results of our numerical experiments and test the performance of the algorithm \hyperref[alg:AGLA]{AGLA} for the STFT-spectrogram inversion. As signals we used the \texttt{© EBU Audio Test sequences}, which provide 70 audio files for testing. To reduce the computational time for each signal, we trimmed down every signal to their first two seconds. In this paper we only present the results of some selected signals of this test set. A reproducible research addendum is available at \url{http://ltfat.org/notes/059/}, where we provide the code, from which one can test different configurations of windows and parameters and the results of our parameter testing. It is important to mention, that the presented algorithm and convergence proofs work for a broad range of transformations. In this publication we restrict ourselves to the application in acoustics using a STFT for $T$. 

The simulations were performed with hop size of 32 and 256 FFT bins and a Gaussian window using the LTFAT toolbox \cite{LTFAT}. This choice of hop size and bins ensures that it is possible to reconstruct the signals uniquely from their measurements \cite{Balan}. {For {more} numerical experiments with a different window function, we refer to \cite{ICASSP}.} As a quality measure we used the \textit{Spectrogram Signal to Noise Ratio}, which is defined as
\begin{align*}
    \text{SSNR}(x)=-10\log_{10}\left(\frac{\norm{\abs{Tx}-s}}{\norm{s}}\right).
\end{align*}
Maximizing the SSNR for a signal $x$ is equivalent to minimizing our objective function \eqref{obj}, since $d_{C_2}=\norm{\abs{\cdot}-s}$ for given $s$ and the negative logarithm is monotonely decreasing.

\begin{table}[]
\centering
    \begin{tabular}{|lllc|}
\hline
\multicolumn{4}{|c|}{\textbf{Best 10 combinations with convergence guarantee}\rule{0pt}{10pt}}                                                    \\ \hline
\multicolumn{1}{|c|}{$\alpha$\rule{0pt}{8pt}} & \multicolumn{1}{c|}{$\beta$} & \multicolumn{1}{c|}{$\gamma$} & \multicolumn{1}{c|}{Average SSNR} \\ \hline
\multicolumn{1}{|l|}{0.09\rule{0pt}{7pt}} & \multicolumn{1}{l|}{1.10}  & \multicolumn{1}{l|}{0.20}  & \multicolumn{1}{c|}{9.75505}    \\ \hline
\multicolumn{1}{|l|}{0.60\rule{0pt}{7pt}}  & \multicolumn{1}{l|}{0.65} & \multicolumn{1}{l|}{0.75} & \multicolumn{1}{c|}{9.71289}    \\ \hline
\multicolumn{1}{|l|}{0.70\rule{0pt}{7pt}}  & \multicolumn{1}{l|}{0.50}  & \multicolumn{1}{l|}{0.70}  & \multicolumn{1}{c|}{9.70320}    \\ \hline
\multicolumn{1}{|l|}{0.19\rule{0pt}{7pt}} & \multicolumn{1}{l|}{1.10}  & \multicolumn{1}{l|}{0.25} & \multicolumn{1}{c|}{9.68508}    \\ \hline
\multicolumn{1}{|l|}{0.28\rule{0pt}{7pt}} & \multicolumn{1}{l|}{1.05} & \multicolumn{1}{l|}{0.20}  & \multicolumn{1}{c|}{9.67243}    \\ \hline
\multicolumn{1}{|l|}{0.22\rule{0pt}{7pt}} & \multicolumn{1}{l|}{1.50}  & \multicolumn{1}{l|}{0.65} & \multicolumn{1}{c|}{9.64285}    \\ \hline
\multicolumn{1}{|l|}{0.14\rule{0pt}{7pt}} & \multicolumn{1}{l|}{1.15} & \multicolumn{1}{l|}{0.30}  & \multicolumn{1}{c|}{9.62835}    \\ \hline
\multicolumn{1}{|l|}{0.33\rule{0pt}{7pt}} & \multicolumn{1}{l|}{1.05} & \multicolumn{1}{l|}{0.25} & \multicolumn{1}{c|}{9.60824}    \\ \hline
\multicolumn{1}{|l|}{0.81\rule{0pt}{7pt}} & \multicolumn{1}{l|}{0.40}  & \multicolumn{1}{l|}{0.65} & \multicolumn{1}{c|}{9.58805}    \\ \hline
\multicolumn{1}{|l|}{0.39\rule{0pt}{7pt}} & \multicolumn{1}{l|}{1.90}  & \multicolumn{1}{l|}{0.90}  & \multicolumn{1}{c|}{9.58760}    \\ \hline
\end{tabular}
\caption{The best parameter combinations satisfying the convergence guarantee with respect to average SSNR after 100 iterations over the test set, disregarding the convergence guarantees}
\label{Table2}
\end{table}

Having three parameters in \hyperref[alg:AGLA]{AGLA} rather than one in \hyperref[alg:FGLA]{FGLA} makes the algorithm on the one hand more flexible and adaptable, but on the other hand makes it more difficult to assert, which parameter combination is the best. Therefore we first present a summary of our numerical parameter tests for the parameters of $\alpha$, $\beta$ and $\gamma$. We took ten signals of the test sequence set and let the algorithms run for 100 iterations, initialized with zero-phase, namely $c_0=s$. The signals were chosen to cover a wide range of acoustical signals, including single instruments, human speech and full orchestras. We tested the performance of parameter combinations, satisfying the sufficient conditions to guarantee convergence, and also the best possible combinations, disregarding the convergence guarantee. It is important to note that these results will vary for different choices of window functions and different redundancies of the STFT. A table including the final SSNR for all tested combinations can be found in the research addendum. 

We computed for $\beta \in [0.1,2]$ and $\gamma \in [0.2,1.6]$ the largest possible $\alpha$ up to two decimal points, which satisfies the conditions \eqref{cond1} and \eqref{cond} and tested the parameter combinations for ten different signals. The results are displayed in Table \ref{Table2}.

%\begin{figure*}
 %   \centering
   % \includegraphics[width=\linewidth]{AGLA1.png}
  %  \caption{The performance of GLA, FGLA and AGLA with zero-phase initialization}
   % \label{fig:Lims}
%\end{figure*}

\begin{table}[]
\centering
\begin{tabular}{|lllc|}
\hline
\multicolumn{4}{|c|}{\textbf{Best 10 combinations over all }\rule{0pt}{10pt}  }                                                                           \\ \hline
\multicolumn{1}{|c|}{$\alpha$\rule{0pt}{8pt}  } & \multicolumn{1}{c|}{$\beta$} & \multicolumn{1}{c|}{$\gamma$} & \multicolumn{1}{c|}{Average SSNR} \\ \hline
\multicolumn{1}{|l|}{1.05 \rule{0pt}{7pt}  }     & \multicolumn{1}{l|}{1.35}    & \multicolumn{1}{l|}{1.25}     & 12.23339                       \\ \hline
\multicolumn{1}{|l|}{0.95\rule{0pt}{7pt}}     & \multicolumn{1}{l|}{1.00}       & \multicolumn{1}{l|}{1.30}      & 12.16774                       \\ \hline
\multicolumn{1}{|l|}{0.95\rule{0pt}{7pt}}     & \multicolumn{1}{l|}{0.99}    & \multicolumn{1}{l|}{1.30}      & 12.08267                       \\ \hline
\multicolumn{1}{|l|}{0.95\rule{0pt}{7pt}}     & \multicolumn{1}{l|}{1.05}    & \multicolumn{1}{l|}{1.30}      & 12.06661                        \\ \hline
\multicolumn{1}{|l|}{1.00\rule{0pt}{7pt}}        & \multicolumn{1}{l|}{1.40}     & \multicolumn{1}{l|}{1.25}     & 11.98671                     \\ \hline
\multicolumn{1}{|l|}{1.05\rule{0pt}{7pt}} & \multicolumn{1}{l|}{1.30}  & \multicolumn{1}{l|}{1.25} & 11.97935 \\ \hline
\multicolumn{1}{|l|}{0.99\rule{0pt}{7pt}} & \multicolumn{1}{l|}{1.00}    & \multicolumn{1}{l|}{1.30}  & 11.89051 \\ \hline
\multicolumn{1}{|l|}{0.99\rule{0pt}{7pt}} & \multicolumn{1}{l|}{0.99} & \multicolumn{1}{l|}{1.30}  & 11.85429 \\ \hline
\multicolumn{1}{|l|}{1.00\rule{0pt}{7pt}}    & \multicolumn{1}{l|}{1.05} & \multicolumn{1}{l|}{1.30}  & 11.85108 \\ \hline
\multicolumn{1}{|l|}{0.99\rule{0pt}{7pt}} & \multicolumn{1}{l|}{1.40}  & \multicolumn{1}{l|}{1.25} & 11.84152 \\ \hline
\end{tabular}
\caption{The best parameter combinations with respect to average SSNR after 100 iterations over the test set}
\label{Table1}
\end{table}

It is interesting to note that all parameter combinations satisfying \eqref{cond1} and \eqref{cond} resulted in a very similar SSNR after 100 iterations. Only for \texttt{Sound 7} the second and third combinations in Table \ref{Table2} converged to a point of significantly larger SSNR. Overall the combinations with $\gamma \approx 0,25$ and $\beta \approx 1,1$ performed over all very well and  $\alpha=0.09$, $\beta=1.1$ and $\gamma=0.2$ was the best combination.

We did the same test considering further 832 combinations with $\alpha \in [0.95,1.25]$, $\beta \in [0.95,1.5]$ and $\gamma \in [0.95,1.3 ]$. Even though these combinations are not covered by the convergence guarantee, they show the best possible convergence of our proposed algorithm \hyperref[alg:AGLA]{AGLA}. In these tests, we noticed that if $\beta$ and $\gamma$ do not fulfill \eqref{cond1}, then the algorithm does not converge. { This observation is included in the research addendum. } It seems optimal to choose $\gamma \approx 1.25$, $\alpha \approx 1$ and $\beta > \alpha$.

We compared \hyperref[alg:AGLA]{AGLA} to two other algorithms, which were observed to perform well in the STFT phase retrieval in \cite{beyond}. The first algorithm is the \emph{Relaxed Averaged Alternating Projections} (\hyperref[alg:RAAR]{RAAR}) proposed by R. Luke for the phase retrieval problem in Diffraction Imaging \cite{Luke}. 
\begin{algorithm}[H]
\caption{Relaxed Averaged Alternating Reflections}\label{alg:RAAR}
\begin{algorithmic}
\STATE 
\STATE {\textsc{Initialize }}$c_0\in \C^n$ and $\lambda \in (0,1]$ \smallskip
\STATE \hspace{0.5cm}$ \textbf{Iterate for } n=1,\dots,N  $
\STATE \hspace{0.5cm}$ c_{n+1}=\frac{\lambda}{2}\left(c_n+R_{C_1}(R_{C_2}(c_n))\right)+(1-\lambda)P_{C_2}(c_n)$\smallskip
\STATE {\textsc{Return }} $T^\dagger c_N$
\end{algorithmic}

\end{algorithm}
According to \cite{Optics} the best performance of \hyperref[alg:RAAR]{RAAR} for speech signals can be expected for $\lambda = 0.9$. 

The second algorithm is the \emph{Difference Map} (\hyperref[alg:DM]{DM}) proposed by V. Elser for phase retrieval in \cite{DM}.

\begin{algorithm}[H]
\caption{Difference Map} \label{alg:DM}
\begin{algorithmic}
\STATE 
\STATE {\textsc{Initialize }}$c_0\in \C^n$ and $\rho \in \R\setminus \{0\}$ \smallskip
\STATE \hspace{0.5cm}$ \textbf{Iterate for } n=1,\dots,N  $
\STATE \hspace{0.5cm}$ t_n = P_{C_2}(c_n)+\frac{1}{\rho}(P_{C_2}(c_n)-c_n)$
\STATE \hspace{0.5cm}$ s_n = P_{C_1}(c_n)+\frac{1}{\rho}(P_{C_1}(c_n)-c_n)$
\STATE \hspace{0.5cm}$ c_{n+1}=c_n+\rho( P_{C_1}(t_n) - P_{C_2}(s_n))$\smallskip
\STATE {\textsc{Return }} $T^\dagger c_N$
\end{algorithmic}

\end{algorithm}

In general choosing $\rho$ close to 1 yields the best performance of \hyperref[alg:DM]{DM} and in \cite{beyond} it was observed that the choice $\rho = 0.8$ performs best. To the best of our knowledge the convergence for \hyperref[alg:DM]{DM} is unproven.

 Based on the observation in Corollary \ref{thm:plot}, we plotted for \hyperref[alg:FGLA]{FGLA} and \hyperref[alg:AGLA]{AGLA} the SSNR of $y_n=P_{C_1}(P_{C_2}(c_n))$ and for \hyperref[alg:RAAR]{RAAR} and \hyperref[alg:DM]{DM} the SSNR of the iterates $c_n$ respectively.

\begin{figure}
    \centering
    \includegraphics[width=0.90\linewidth]{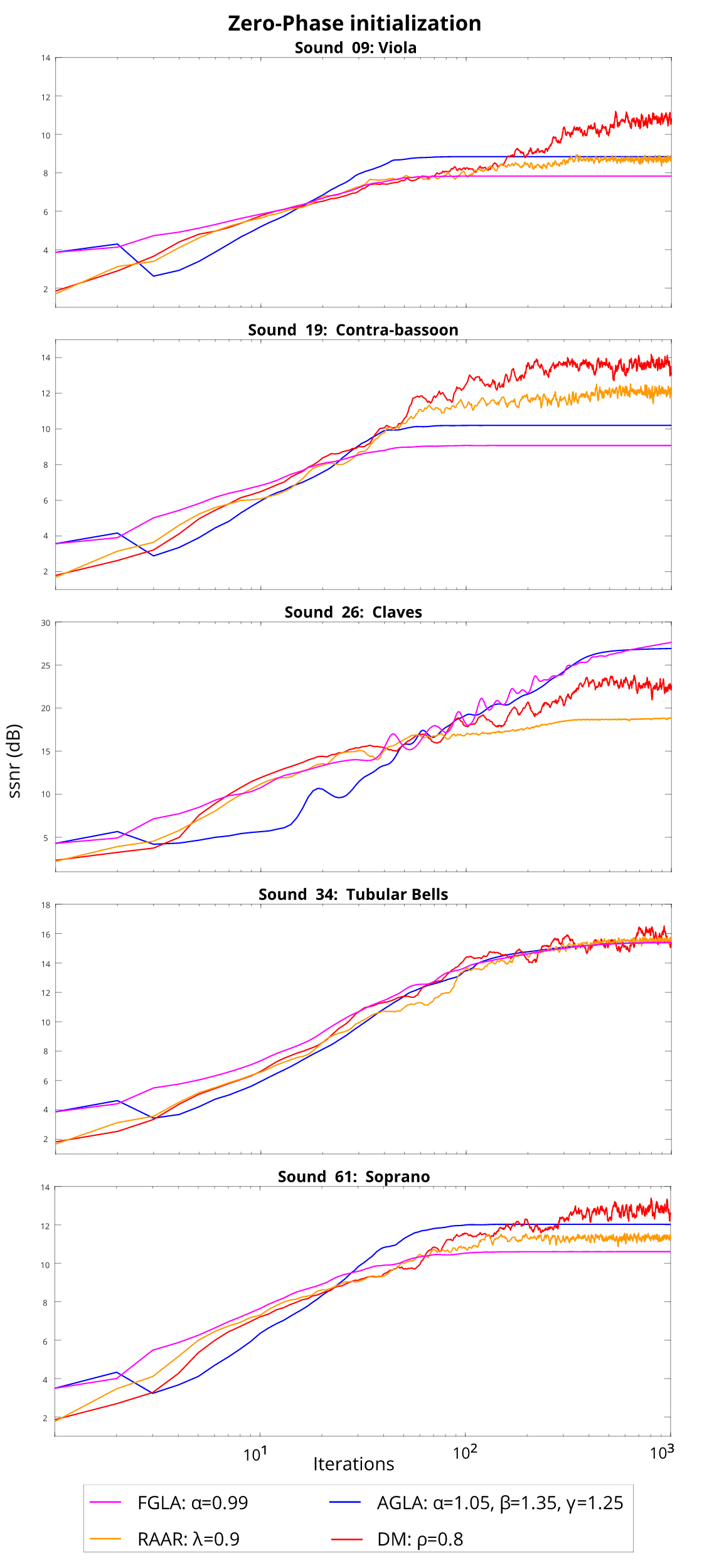}
    \caption{Comparison of \hyperref[alg:FGLA]{FGLA}, \hyperref[alg:AGLA]{AGLA}, \hyperref[alg:RAAR]{RAAR}, and \hyperref[alg:DM]{DM} with their respective best parameter choice with zero-phase initialization}
    \label{ZERO}
    \vspace*{-4mm}
\end{figure}

%Per signal every parameter combination was rated by a scale from 0-1, depending on how close its final SSNR is comapred to the best achieved SSNR for this signal. We averaged this points over all X signals. This means that the combination $\alpha=$, $\beta= $ and $\gamma = $ reached an SSNR, which was Y% smaller than the best one. We clearly see that this combination stands out from the rest. 

%Now we will compare GLA to AGLA and FGLA for the parameter choices which guarantee convergence.

%\textbf{[Results parameter search (half a page)]}

For the comparison between the algorithms, we started by comparing \hyperref[alg:AGLA]{AGLA}, \hyperref[alg:FGLA]{FGLA}, \hyperref[alg:DM]{DM}, and \hyperref[alg:RAAR]{RAAR} for 1000 iterations with zero-phase initialization, namely, taking $c_0=s$ without any phases. For these simulations we used signals which were not included in the parameter search to remove any bias. For the purpose of a better overview, we did not include the algorithm \hyperref[alg:GLA]{GLA}, since the observations in \cite{FGLA} clearly showed that \hyperref[alg:FGLA]{FGLA} outperforms \hyperref[alg:GLA]{GLA} in nearly every setting.  

For the Figure \ref{ZERO} we considered five signals, varying in their nature. For \hyperref[alg:DM]{DM} and \hyperref[alg:RAAR]{RAAR} we chose $\rho=0.8$ and $\lambda=0.9$  respectively, for \hyperref[alg:FGLA]{FGLA} $\alpha=0.99$ and for \hyperref[alg:AGLA]{AGLA} the best combination from Table \ref{Table1}. 

\begin{figure}[htp]
    \centering
    \includegraphics[width=0.90\linewidth]{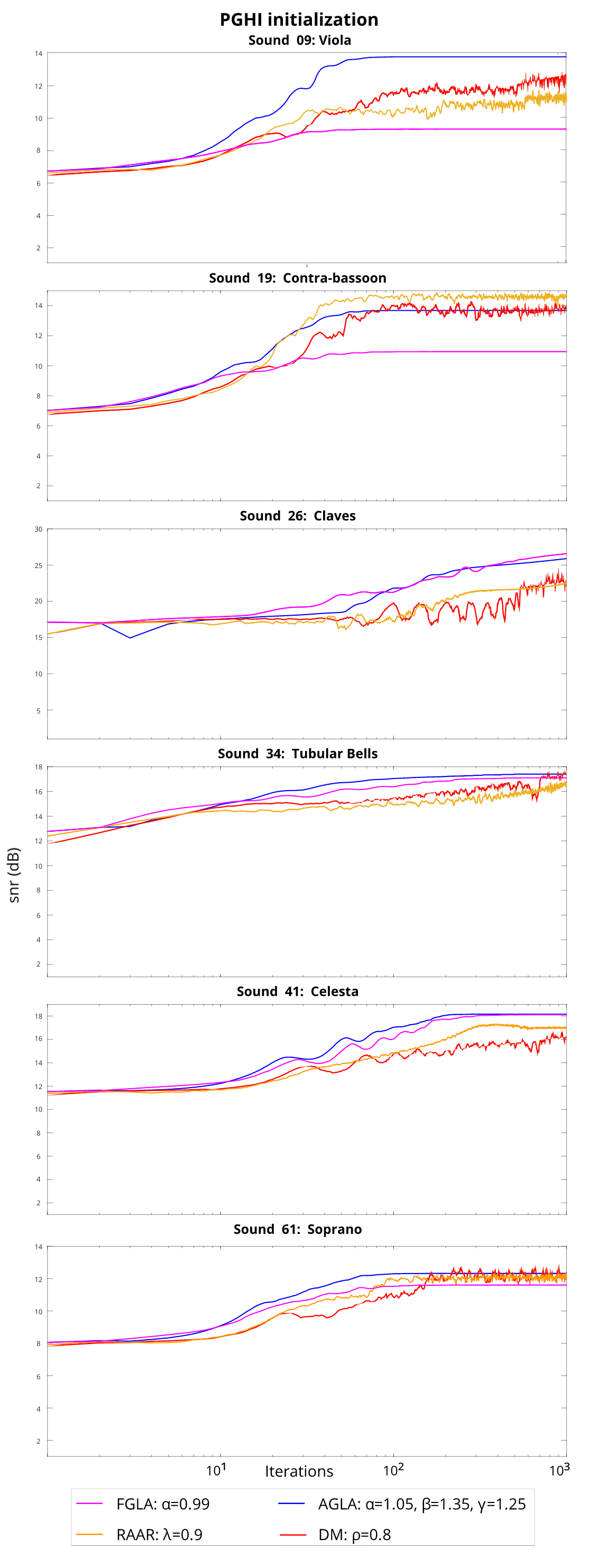}
    \caption{Comparison of \hyperref[alg:FGLA]{FGLA}, \hyperref[alg:AGLA]{AGLA}, \hyperref[alg:RAAR]{RAAR}, and \hyperref[alg:DM]{DM} with their respective best parameter choice with PGHI initialization}
    \vspace*{-4mm}
    \label{PGHI}
\end{figure}

We can observe that \hyperref[alg:DM]{DM} performed over all best, but it is also important to note that \hyperref[alg:DM]{DM} computes twice as many different projections per iterations compared to the other algorithms and the projections are the computationally most expensive segments of the algorithms. Still we see that \hyperref[alg:AGLA]{AGLA} can outperform \hyperref[alg:DM]{DM} (\texttt{Sound 26}) or reach a very similar final SSNR (\texttt{Sound 34}, \texttt{41} and \texttt{61}). Furthermore \hyperref[alg:RAAR]{RAAR}, an algorithm commonly used for the phase retrieval problem in areas beyond acoustics \cite{Luke}, performed for \texttt{Sound 19} better than \hyperref[alg:AGLA]{AGLA} and for the other four it reached a point of similar or lower SSNR. \hyperref[alg:DM]{DM} and \hyperref[alg:RAAR]{RAAR} started oscillating after 100 iterations, whereas \hyperref[alg:FGLA]{FGLA} and \hyperref[alg:AGLA]{AGLA} stayed stable, an important trait in algorithms. \hyperref[alg:AGLA]{AGLA} achieved higher SSNR than \hyperref[alg:FGLA]{FGLA} over all, except for \texttt{Sound 26}, where both of them performed similar and better than \hyperref[alg:DM]{DM} and \hyperref[alg:RAAR]{RAAR}.
%Besides SSNR we considered PEMO-Q to evaluate the quality of the reconstructed signal \cite{PEMOQ}. This method is based on a psychoacoustically validated, quantitative model of the "effective" peripheral auditory and predicts the perceived audio quality. In the following experiment we computed the perceptual similarity measure (PSM) for 30 different signals, which were not included in the parameter search. According to \cite{PEMOQ} PSM correlates well with subjective quality ratings \tcr{if different types of audio signals are considered separately.}

For the next experiment we initialized the algorithms with the reconstructed signal from the \emph{Phase Gradient Heap Integration} (PGHI) presented in \cite{PGHI} and compared them in Figure \ref{PGHI} for the same signals as in Figure \ref{ZERO}. 

PGHI is a noniterative method for the phase retrieval problem, based on the phase-magnitude relations of an continuous STFT. It is efficient, but when used for the discrete setting it introduced inaccuracies depending on the parameters of the discrete STFT, and it was observed to give a good starting point for iterative algorithms \cite{Nicki}.

With PGHI-initialization we observe that the starting SSNR of the algorithms was significantly higher and notice that \hyperref[alg:AGLA]{AGLA} improved the most. This could be explained by the design of \hyperref[alg:AGLA]{AGLA}, as mentioned before, enabling the iterates to escape local solution as long as the steps $t_{n}-t_{n-1}$ and the difference between the direction of the projected and nonprojected sequences are large enough. \hyperref[alg:AGLA]{AGLA} performed better or equal than \hyperref[alg:DM]{DM} on all five signals and \hyperref[alg:RAAR]{RAAR} only reached a point of higher SSNR for \texttt{Sound 19}, similar to the zero-phase initialization experiment. For \hyperref[alg:DM]{DM} and \hyperref[alg:RAAR]{RAAR} there was a slight increase in the intensity of the oscillations. Interestingly for \texttt{Sound 61}, \hyperref[alg:FGLA]{FGLA}, \hyperref[alg:RAAR]{RAAR} and \hyperref[alg:DM]{DM} performed slightly worse than with zero-phase initialization, and only \hyperref[alg:AGLA]{AGLA} reached the same SSNR as before. This suggest that \hyperref[alg:AGLA]{AGLA} is an algorithm suitable for hybrid schemes, where one initializes with either a non-iterative scheme or the end point of another scheme, maybe faster, after a fixed number of iterations as it has been studied in \cite{beyond}.

In Figure \ref{LIMS} we plotted \hyperref[alg:GLA]{GLA} with \hyperref[alg:FGLA]{FGLA} and \hyperref[alg:AGLA]{AGLA} with their respective best parameter choices covered by the convergence guarantee. We notice that \hyperref[alg:AGLA]{AGLA} had some oscillation in the first 20 iterations and after that catches up to \hyperref[alg:GLA]{GLA} and \hyperref[alg:FGLA]{FGLA}, surpassing them in the end. This behaviour could be explained by $\beta=1.1$ being relatively large and $\gamma=0.2$ being relatively small, which means that the inertial/momentum step of the nonprojected sequence can be quite large and the projected iterates are weighted less, therefore we experienced not so optimal performance in the beginning, until the iterates got closer to a local solution. %Except for \texttt{Sound 26}, \hyperref[alg:FGLA]{FGLA} reached a similar SSNR as \hyperref[alg:GLA]{GLA}, but with a greater speed. 

%%%%%%%%%%%

\section{Conclusion}
In this paper we presented the \textit{Accelerated Griffin-Lim algorithm} and proved convergence results for it and its predecessor, the \textit{Fast Griffin-Lim algorithm}. If the parameters are chosen to fulfill the necessary conditions to guarantee convergence and the minimizing properties, both {of them} outperform the \textit{Griffin-Lim algorithm}, making it now possible for them to replace {this classical method} as the standard and reliable phase retrieval algorithm for acoustic. We showed that one can expect good convergence behaviour of inertial based methods theoretically and practically in the phase retrieval setting. The numerical results indicate that there are parameter combinations outside of {the convergence} regimes for which the algorithm asserts fast behaviour, further sparking interest in studying these methods. The simulations indicate that the \textit{Accelerated Griffin-Lim algorithm} has the possibility to perform similarly error-wise to \textit{Relaxed Averaged Alternating Reflections} and \textit{Difference Map}, the two other powerful retrieval methods. The 
%\textit{Accelerated Griffin-Lim algorithm} 
proposed method {achives} better convergence behaviour {in the sense of having fewer to none oscillations and performing better for good initialization.} 

\begin{figure}[h]
    \centering
    \includegraphics[width=0.90\linewidth]{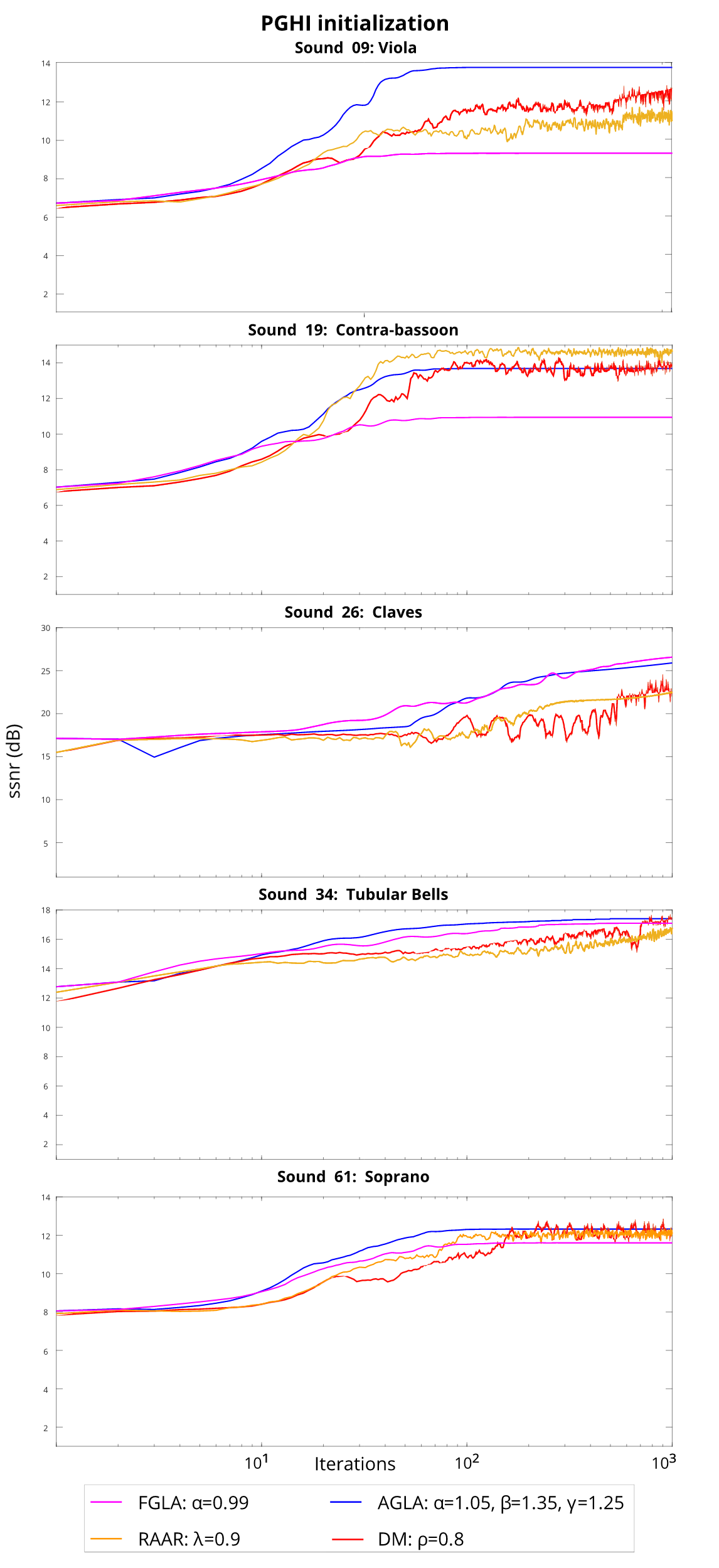}
    \caption{Comparison of \hyperref[alg:GLA]{GLA}, \hyperref[alg:FGLA]{FGLA}, and \hyperref[alg:AGLA]{AGLA}, with their respective parameter choice satisfying the convergence guarantee}
    \label{LIMS}
    \vspace*{-4mm}
\end{figure}

%\section{Acknowledgement}
%This work was supported by the Austrian Science Fund FWF-project NoMASP (“Nonsmooth nonconvex optimization methods for acoustic signal processing”; P 34922-N). The authors would like to express their gratitude to Dr. Nicki Holighaus (Austria Academy of Sciences) for his valuable input and ideas in the numerical experiments section.

%%%%%%%%%%%%%%%

\vfill
{\footnotesize
\bibliographystyle{IEEEtran}

\begin{thebibliography}{10}
\providecommand{\url}[1]{#1}
\csname url@samestyle\endcsname
\providecommand{\newblock}{\relax}
\providecommand{\bibinfo}[2]{#2}
\providecommand{\BIBentrySTDinterwordspacing}{\spaceskip=0pt\relax}
\providecommand{\BIBentryALTinterwordstretchfactor}{4}
\providecommand{\BIBentryALTinterwordspacing}{\spaceskip=\fontdimen2\font plus
\BIBentryALTinterwordstretchfactor\fontdimen3\font minus
  \fontdimen4\font\relax}
\providecommand{\BIBforeignlanguage}[2]{{%
\expandafter\ifx\csname l@#1\endcsname\relax
\typeout{** WARNING: IEEEtran.bst: No hyphenation pattern has been}%
\typeout{** loaded for the language `#1'. Using the pattern for}%
\typeout{** the default language instead.}%
\else
\language=\csname l@#1\endcsname
\fi
#2}}
\providecommand{\BIBdecl}{\relax}
\BIBdecl

\bibitem{ac1}
K.~Jaganathan, Y.~C. Eldar, and B.~Hassibi, ``Stft phase retrieval: Uniqueness
  guarantees and recovery algorithms,'' \emph{IEEE J Sel Top}, vol.~10, no.~4,
  pp. 770--781, 2016.

\bibitem{ac2}
R.~A. Bedoui, Z.~Mnasri, and F.~Benzarti, ``Phase retrieval: Application to
  audio signal reconstruction,'' in \emph{2022 19th International
  Multi-Conference on Systems, Signals \& Devices (SSD)}, 2022, pp. 21--30.

\bibitem{imag1}
O.~Yurduseven, D.~R. Smith, and T.~Fromenteze, ``Phase retrieval in
  frequency-diverse imaging,'' in \emph{IEEE Int. Symp. Antennas Propag. \&
  USNC/URSI National Radio Sci. Meet.}, 2018, pp. 1797--1798.

\bibitem{imag2}
Y.~Shechtman, Y.~Eldar, O.~Cohen, H.~Chapman, J.~Miao, and M.~Segev, ``Phase
  retrieval with application to optical imaging,'' \emph{IEEE Signal Process.
  Mag.}, vol.~32, 02 2014.

\bibitem{elec1}
P.~Bowlan and R.~Trebino, ``Phase retrieval and the measurement of the complete
  spatiotemporal electric field of ultrashort pulses,'' in \emph{CLEO/QELS:
  Laser Sci. to Photonic Appl.}, 2010, pp. 1--2.

\bibitem{elec2}
M.~Johansson, H.-S. Lui, A.~Fhager, and M.~Persson, ``Electromagnetic source
  modeling using phase retrieval methods,'' in \emph{XXXth URSI Gen. Assem.
  Sci. Symp.}, 2011, pp. 1--2.

\bibitem{Nicki}
A.~Marafioti, N.~Holighaus, and P.~Majdak, ``Time-frequency phase retrieval for
  audio—the effect of transform parameters,'' \emph{IEEE Trans. Signal
  Process.}, vol.~69, pp. 3585--3596, 2021.

\bibitem{GLA}
D.~Griffin and J.~Lim, ``Signal estimation from modified short-time fourier
  transform,'' \emph{IEEE Trans. Acoust. Speech Signal Process.}, vol.~32, pp.
  236 -- 243, 1984.

\bibitem{GS}
R.~Gerchberg and W.~Saxton, ``A practical algorithm for the determination of
  phase from image and diffraction plane pictures,'' \emph{Optik}, vol.~35, pp.
  237--250, 1971.

\bibitem{FISTA}
A.~Beck and M.~Teboulle, ``A fast iterative shrinkage-thresholding algorithm
  for linear inverse problems,'' \emph{SIAM J. Imaging Sci.}, vol.~2, no.~1,
  pp. 183--202, 2009.

\bibitem{FGLA}
N.~Perraudin, P.~Balazs, and P.~Søndergaard, ``A fast griffin–lim
  algorithm,'' \emph{IEEE Workshop Appl. Signal Process. Audio Acoust.}, pp.
  1--4, 2013.

\bibitem{10.5555/2815664}
M.~M{\"u}ller, \emph{Fundamentals of Music Processing: Audio, Analysis,
  Algorithms, Applications}, 1st~ed.\hskip 1em plus 0.5em minus 0.4em\relax
  Springer Publishing Company, Incorporated, 2015.

\bibitem{9099086}
K.~Yatabe, ``Consistent ica: Determined bss meets spectrogram consistency,''
  \emph{IEEE Signal Processing Letters}, vol.~27, pp. 870--874, 2020.

\bibitem{Marafiotia}
A.~Marafioti, P.~Majdak, N.~Holighaus, and N.~Perraudin, ``Gacela - a
  generative adversarial context encoder for long audio inpainting,''
  \emph{IEEE Journal of Selected Topics in Signal Processing}, in press.

\bibitem{SALEEM2022104389}
N.~Saleem, J.~Gao, M.~Irfan, E.~Verdu, and J.~P. Fuente, ``E2e-v2sresnet: Deep
  residual convolutional neural networks for end-to-end video driven speech
  synthesis,'' \emph{Image and Vision Computing}, vol. 119, p. 104389, 2022.

\bibitem{FrameTh}
O.~Christensen, \emph{An Introduction to Frames and Riesz Bases}.\hskip 1em
  plus 0.5em minus 0.4em\relax Birkh\"{a}user Cham, 2016.

\bibitem{Balan}
R.~Balan, P.~Casazza, and D.~Edidin, ``On signal reconstruction from absolute
  value of frame coefficients,'' \emph{Proceedings of SPIE}, 2005.

\bibitem{Grohs}
P.~Grohs, S.~Koppensteiner, and M.~Rathmair, ``Phase retrieval: Uniqueness and
  stability,'' \emph{SIAM Review}, 2020.

\bibitem{Pole}
P.~Balazs, D.~Bayer, F.~Jaillet, and P.~Søndergaard, ``The pole behaviour of
  the phase derivative of the short-time fourier transform,'' \emph{Applied and
  Computational Harmonic Analysis}, vol.~40, 2015.

\bibitem{HBF}
B.~Polyak, ``Some methods of speeding up the convergence of iteration
  methods,'' \emph{USSR Computational Mathematics and Mathematical Physics},
  vol.~4, no.~5, pp. 1--17, 1964.

\bibitem{Bot-Csetnek-Laszlo}
R.~I. Bo\c{t}, E.~R. Csetnek, and S.~C. Laszlo, ``An inertial
  forward–backward algorithm for the minimization of the sum of two nonconvex
  functions,'' \emph{EURO J. Comput. Optim.}, vol.~4, p. 3–25, 2016.

\bibitem{Laszlo}
S.~C. Laszlo, ``Forward-backward algorithms with different inertial terms for
  structured non-convex minimization problems,'' \emph{J. Optim. Theory Appl.},
  2023.

\bibitem{rockafellar}
R.~Rockafellar and R.~Wets, \emph{Variational analysis}, ser. Grundlehren Math.
  Wiss.\hskip 1em plus 0.5em minus 0.4em\relax Springer, 2011, vol. 317.

\bibitem{Attouch-Bolte}
H.~Attouch and J.~Bolte, ``On the convergence of the proximal algorithm for
  nonsmooth functions involving analytic features,'' \emph{Math. Program.},
  vol. 116, p. 5–16, 22009.

\bibitem{Bolte-Sabach-Teboulle}
J.~Bolte, S.~Sabach, and M.~Teboulle, ``Proximal alternating linearized
  minimization for nonconvex and nonsmooth problems,'' \emph{Math. Program.},
  vol. 146, 08 2013.

\bibitem{SemiAlg}
H.~Attouch, J.~Bolte, and B.~Svaiter, ``Convergence of descent methods for
  semi-algebraic and tame problems: Proximal algorithms, forward-backward
  splitting, and regularized gauss-seidel methods,'' \emph{Math. Program.},
  vol. 137, pp. 91--129, 2013.

\bibitem{LiPong}
G.~Li and T.~K. Pong, ``Calculus of the exponent of kurdyka-\l{}ojasiewicz
  inequality and its applications to linear convergence of first-order
  methods,'' \emph{Found. Comput. Math.}, vol.~18, no.~5, pp. 1199--1232, 2018.

\bibitem{LTFAT}
Z.~Pr{\r{u}}{\v{s}}a, P.~L. S{\o}ndergaard, N.~Holighaus, C.~Wiesmeyr, and
  P.~Balazs, ``The large time-frequency analysis toolbox 2.0,'' in \emph{Sound,
  Music, and Motion}, M.~Aramaki, O.~Derrien, R.~Kronland-Martinet, and
  S.~Ystad, Eds.\hskip 1em plus 0.5em minus 0.4em\relax Cham: Springer
  International Publishing, 2014, pp. 419--442.

\bibitem{ICASSP}
R.~Nenov, D.-K. Nguyen, and P.~Balazs, ``Faster than fast: Accelerating the
  griffin-lim algorithm,'' \emph{IEEE Int. Conf. Acoust. Speech Signal Process
  (ICASSP)}, pp. 1--5, 2023.

\bibitem{beyond}
T.~Peer, S.~Welker, and T.~Gerkmann, ``Beyond griffin-lim: Improved iterative
  phase retrieval for speech,'' in \emph{Int. Workshop Acoustic Signal
  Enhancement (IWAENC)}, 2022, pp. 1--5.

\bibitem{Luke}
D.~R. Luke, ``Relaxed averaged alternating reflections for diffraction
  imaging,'' \emph{Inverse Probl.}, vol.~21, 2004.

\bibitem{Optics}
T.~Kobayashi, T.~Tanaka, K.~Yatabe, and Y.~Oikawa, ``Acoustic application of
  phase reconstruction algorithms in optics,'' \emph{IEEE Int. Conf. Acoust.
  Speech Signal Process (ICASSP)}, pp. 6212--6216, 2022.

\bibitem{DM}
V.~Elser, ``Phase retrieval by iterated projections,'' \emph{J. Opt. Soc. Am.
  A}, vol.~20, pp. 40--55, 2003.

\bibitem{PGHI}
Z.~Průša, P.~Balazs, and P.~Søndergaard, ``A noniterative method for
  reconstruction of phase from stft magnitude,'' \emph{IEEE/ACM IEEE Trans.
  Audio Speech Language Process.}, vol.~25, pp. 1154--1164, 05 2017.

\bibitem{DistSqu}
G.~Garrigos, ``Square distance functions are polyak-Łojasiewicz and
  vice-versa,'' \emph{arXiv: 2301.10332}, 2023.

\end{thebibliography}

\vfill

\appendix

\begin{proof}[Proof of Proposition \ref{d2}]
For given $a\in \C$ and $r>0$, one can deduce by rewriting in polar coordinates that
\begin{align}
    \min_{\abs{b}=r} \abs{a-b}^2&=\min_{\theta\in [0,2\pi)} \abs{ a-re^{i\theta}}^2\label{app:min}  \\
    &= \min_{\theta\in [0,2\pi)} \abs{ a}^2+r^2-2\abs{a}r\cos (\angle a - \theta ). \notag
\end{align}
Since this minimum is attained at $\theta=\angle a$, we get
\begin{align*}
    \min_{\abs{b}=r} \abs{a-b}^2&= \abs{a-re^{i\angle a}}^2\\
    &=  \abs{\abs{a}-r}^2 \abs{e^{i \angle a}}^2
    =\abs{\abs{a}-r}^2.
\end{align*}
If $r=0$, then $b=0$ is the unique minimizer of \eqref{app:min}. Therefore we can see that
\begin{align}
    d_{C_2}^2(c)&=\min_{v\in C_2} \norm{c-v}^2 \label{app:minsum} \\ &= \sum_{i=1}^L\min_{v_i\in \C, \abs{v_i}=s_i}\abs{c_i-v_i}^2 \notag\\&= \sum_{i=1}^L\abs{\abs{c_i}-s_i}^2=\norm{\abs{c}-s}^2, \notag
\end{align}
which proves the first statement. For the second statement, we use the definition \eqref{def:PC2} of $P_{C_2}$ 
\begin{align*}
    \norm{c-P_{C_2}(c)}^2&= \sum_{i=1}^L \abs{c_i-(P_{C_2}(c))_i}^2\\
    &=\sum_{c_i\neq 0 }\abs{c_i-c_i\frac{s_i}{\abs{c_i}}}^2+\sum_{ c_i= 0 }\abs{s_i}^2 \\
    &=\sum_{c_i\neq 0 }\abs{c_i}^2\abs{1-\frac{s_i}{\abs{c_i}}}^2+\sum_{ c_i= 0 }\abs{s_i}^2
    \\
    &=\sum_{c_i\neq 0 }\abs{\abs{c_i}-s_i}^2+\sum_{ c_i= 0 }\abs{s_i}^2 \\
    &=\sum_{i=1}^L\abs{\abs{c_i}-s_i}^2 = d^2_{C_2}(c).
\end{align*}
Moreover, we see from \eqref{app:min} that $\theta = \angle \alpha$ is the unique minimizer when $a\neq 0$, thus $P_{C_2}(c)$ is the unique minimizer of \eqref{app:minsum} for $c\notin D$ by the calculations above. 
\end{proof}

\begin{proof}[Proof of Proposition \ref{prop:Subdiff}]
In \cite{DistSqu}, an explicit formula for the limiting subgradient of the power of a distance function is derived. By this result, we know that for $y\in \R^{M \times 2}$
\begin{align*}
    \partial \left( \frac{1}{2}d_{C_2}^2\right)(y)={y}-\overline{P}_{C_2}(y).
\end{align*}
By Proposition \ref{d2} we can deduce that for $y\notin D$
\begin{align*}
    \partial \left(\frac{1}{2}d_{C_2}^2\right)(y)= \{ y-P_{C_2}(y) \}
\end{align*}
holds, since then $\overline{P}_{C_2}(y)=\{P_{C_2}(y)\}$ and that $\frac{1}{2}d_{C_2}^2(y)=\frac{1}{2}\norm{y-P_{C_2}(y)}^2.$ {If we look at the definition of $P_{C_2}$, we notice that each component is continuously differentiable around a neighbourhood of any $y\notin D$, since $x\mapsto \frac{x}{\abs{x}}$ is continously differentiable around a neighbourhood of any $x\neq 0$. Since $\norm{\cdot}^2$ is continuously differentiable everywhere, we deduce that $\frac{1}{2}d_{C_2}^2$ is continuously differentiable  around a neighbourhood of any $y\notin D$.} Therefore $\partial \frac{1}{2}d_{C_2}^2(y)=\nabla \frac{1}{2}d_{C_2}^2(y)$ for $y\notin D$ by \cite[Corrolary 9.19]{rockafellar}.

\color{black}

For $\delta_{C_1}$ the limiting subgradient is given by the orthogonal space $\partial \delta_{C_1}(y)=C_1^\perp$ for $y\in C_1$ by \cite[Exercise 8.14]{rockafellar}. Furthermore, using the sum rule forumla \cite[Exercise 8.8]{rockafellar}, we can deduce for the objective function $f(y)=\delta_{C_1}(y)+\frac{1}{2}d_{C_2}^2(y)$ that for $y \in C_1 \setminus D$
\begin{align*}
\partial f(y)=C_1^\perp + y-P_{C_2}(y) .
\end{align*}

By the definition of the distance of the sum of sets to a point, we have that for $y\in C_1\setminus D$
\begin{align}
   d_{\partial f(y)}(0)&\leq \min_{z\in C_1^\perp,u\in \overline{P}_{C_2}(y)} \norm{z+y-u} \notag\\
   &\leq \min_{z\in C_1^\perp} \norm{z+y-P_{C_2}(y)}\label{app:d0}
\end{align}
By taking $z=P_{C_2}(y)-P_{C_1}(P_{C_2}(y)) \in C_1^\perp$, we get the conclusion from \eqref{app:d0}.
\end{proof}

\color{black}

%\begin{proof}[Proof of Proposition  \ref{optimality}]
%    As mentioned in the proof above $\frac{1}{2}d^2_{C_2}$ is {smooth} around a neighbourhood of $c^*$ if $c^*\notin D$. By \cite[Theorems 10.1 and 8.2]{rockafellar}, we have that $c^*$ is a local minimum of $f(y)=\frac{1}{2}d_{C_2}^2(y)+ \partial \delta_{C_1}(y)$ if {and only if}
 %   \begin{align*}
  %      {-} \skal{\nabla \frac{1}{2}d_{C_2}^2(c^*),y}\leq \delta_{C_1}(y) \quad \text{ for all }y\in \R^{M \times 2}.
    %\end{align*}
   % This inequality holds trivially for $y\notin C_1$ as $\partial \delta_{C_1}(y)=+\infty$. In Theorem \ref{thm:Main} we proved $c^*-P_{C_1}(P_{C_2}(c^*))=0$, therefore 
    %\begin{align*}
   %     {-} \nabla \frac{1}{2}d_{C_2}^2(c^*)={P_{C_2}(c^*)-c^*}\in C_1^\perp.
  %  \end{align*}
 %   This proves the inequality for $y\in C_1,$ since in this case both sides are zero.
%\end{proof}

%%%%%%%%%%%%%%%%%%%%

\begin{proof}[Proof of Proposition  \ref{prop:Freg}]
 Choose $m\in \N$ such that for all $n\geq m$ the vector $c_n\notin D$ and let $n\geq m$. By the decreasing property of Theorem \ref{thm:Main} and by \eqref{i ct} we know that
\begin{align*}
    F_{K_2}(c_n,t_n) + \kappa_1\norm{c_n-t_n}^2 \leq F_{K_2}(c_{n-1},t_{n-1}),
\end{align*}
where $\kappa_1=\frac{K_1-K_2}{2\alpha^2}$ and thus proving the first statement, since $K_1>K_2$. \smallskip 
Using the triangle inequality we can see that by \eqref{eq:subdifF} and a similar argument as in Proposition \ref{prop:Subdiff} 
\begin{multline}\label{app2}
    d_{\partial F_n}(0)\leq \norm{c_n-P_{C_1}(P_{C_2}(c_n))+K_2(c_n-t_n)}\\+K_2\norm{c_n-t_n},
\end{multline}
since $c_n \in C_1\setminus D$. By \eqref{PR1.5} we know that %% Could be Replaced with Re
\begin{align*}
    c_n-P_{C_1}(P_{C_2}(c_n))= \frac{1}{\gamma} (c_n-t_{n+1})+\frac{\gamma-1}{\gamma}(c_n-d_{n})
\end{align*}
Furthermore, by the definition of the algorithm we see that $c_n-d_n=\frac{\alpha}{\alpha-\beta}(c_n-t_{n})$.
Combining this observations in \eqref{app2} we see that
\begin{align*}
    d_{\partial F_n}(0) \leq \frac{1}{\gamma}\norm{c_n-t_{n+1}}+\mu\norm{c_n-t_n}
\end{align*}
with $\mu=\abs{ \frac{(\gamma-1)\alpha}{\gamma(\alpha-\beta)}+K_2}+K_2>0$. Furthermore using the triangle inequality we see that
\begin{align*}
    d_{\partial F_n}(0) \leq \frac{1}{\gamma}\norm{t_n-t_{n+1}}+\left(\frac{1}{\gamma}+\mu \right) \norm{c_n-t_n}.
\end{align*}
Since $c_n-t_n=\alpha(t_n-t_{n-1})$ we conclude that
\begin{align*}
     d_{\partial F_n}(0) \leq \kappa_2(\Delta_{t_{n+1}}+\Delta_{t_n})
\end{align*}
with $\kappa_2=\max\{\frac{1}{\gamma},{(\frac{1}{\gamma}+\mu) \alpha} \}$.
\end{proof}

\begin{lem} \label{lem:para}
Suppose that $\gamma >0$ and
\begin{equation}
\label{para:beta}
    0 \leq  2\beta{ \left\lvert 1 - \gamma \right\rvert}  <  {2 - \gamma}
\end{equation}
(i) For $0 < \gamma \leq 1$, if
\begin{equation}  \label{alpha:0-1}
    0 < \alpha < \left( 1 - \dfrac{1}{\gamma} \right) \beta + \dfrac{1}{\gamma} - \dfrac{1}{2} 
\end{equation}
then
\begin{equation*}
    K_1 > K_2 > 0 ,
\end{equation*}
where $K_1 := \frac{1-\gamma}{\gamma}(1+2\alpha+\alpha^2-\beta-\alpha\beta)+\frac{1}{\gamma}(1-\alpha-\alpha^2)$ and $K_2 := \frac{1-\gamma}{\gamma}(\alpha^2+\beta-\alpha \beta)+\frac{1}{\gamma}(\alpha-\alpha^2)$.
% \begin{align*}
%     K_1 & := \frac{1-\gamma}{\gamma}(1+2\alpha+\alpha^2-\beta-\alpha\beta)+\frac{1}{\gamma}(1-\alpha-\alpha^2), \\
%     K_2 & := \frac{1-\gamma}{\gamma}(\alpha^2+\beta-\alpha \beta)+\frac{1}{\gamma}(\alpha-\alpha^2).
% \end{align*}

\noindent
(ii) For $1 < \gamma < 2$, if
\begin{equation}
\label{alpha:1-2}
    0 <  \alpha < \frac{1}{ 2\beta(\gamma-1) + \gamma  } - \frac{1}{2}
\end{equation}
then
\begin{equation*}
    K_1 > K_2 > 0 ,
\end{equation*}
where $K_1 := \frac{1-\gamma}{\gamma}(1+2\alpha+\alpha^2+\beta+\alpha\beta)+\frac{1}{\gamma}(1-\alpha-\alpha^2)$ and $K_2 := \frac{1-\gamma}{\gamma}(\alpha^2-\beta-3\alpha \beta)+\frac{1}{\gamma}(\alpha-\alpha^2)$.
% \begin{align*}
%     K_1 & := \frac{1-\gamma}{\gamma}(1+2\alpha+\alpha^2+\beta+\alpha\beta)+\frac{1}{\gamma}(1-\alpha-\alpha^2), \\
%     K_2 & := \frac{1-\gamma}{\gamma}(\alpha^2-\beta-3\alpha \beta)+\frac{1}{\gamma}(\alpha-\alpha^2).
% \end{align*}
\end{lem}

\begin{proof}
One can check that \eqref{para:beta} guarantees the the feasibility of $\alpha$ in both \eqref{alpha:0-1} and \eqref{alpha:1-2}.
%When $\gamma = 1$, then it is straightforward to see that the conditions reduces to $\beta > 0$ and $0 < \alpha < \frac{1}{2}$.    
%Now for $0 < \gamma < 1$, the condition \eqref{para:beta} guarantees that $2 - \gamma + 2 \left( \gamma - 1 \right) \beta > 0$ and furthermore $K_1 > K_2$ as long as \eqref{alpha:0-1} holds.
Moreover, \eqref{alpha:0-1} implies 
\begin{equation*}
    2 \gamma \alpha < 2 - \gamma + 2 \left( \gamma - 1 \right) \beta,
\end{equation*}
which is equivalent to $K_1-K_2>0$. After some calculations, we see that $K_2 > 0$ is equivalent to
\begin{equation}
\label{alpha:quad:1}
    \gamma \alpha^{2} - \left( 1 + \left( \gamma - 1 \right) \beta \right) \alpha + \left( \gamma - 1 \right) \beta < 0 .
\end{equation}
Observe that for every $0 < \gamma < 1$ it holds
\begin{align*}
    \Delta_{1} & = \left( 1 + \left( \gamma - 1 \right) \beta \right) ^{2} - 4 \gamma \left( \gamma - 1 \right) \beta \\
    & > \left( 1 + \left( \gamma - 1 \right) \beta \right) ^{2} > 0 ,
\end{align*}
Hence the inequality \eqref{alpha:quad:1} holds for $\alpha>0$ if and only if
\begin{equation*}
    0< \alpha < \dfrac{1 + \left( \gamma - 1 \right) \beta + \sqrt{\Delta_{1}}}{2 \gamma} .
\end{equation*}
Furthermore, if $\alpha $ fulfills \eqref{alpha:0-1}, then for $0 < \gamma < 1$
\begin{align*}
    2\gamma\alpha&<2 - \gamma + 2 \left( \gamma - 1 \right) \beta \\
    & < 2 + 2 \left( \gamma - 1 \right) \beta \nonumber \\
    & < 1 + \left( \gamma - 1 \right) \beta + \left\lvert 1 + \left( \gamma - 1 \right) \beta \right\rvert \nonumber \\
    & < 1 + \left( \gamma - 1 \right) \beta + \sqrt{\Delta_{1}} .
\end{align*}
Therefore \eqref{alpha:0-1} implies $K_1>K_2>0$ for this case. 
The result for $1 < \gamma < 2$ can be deduced similarly.
%\eqref{alpha:1-2} that \begin{equation*}
 %    \left( 2 \gamma + 4 \left( \gamma - 1 \right) \beta \right) \alpha < 2 - \gamma + 2 \left( 1 - \gamma \right) \beta,
%\end{equation*} which implies that $K_1 > K_2$.
%After some calculation, we see that $K_2>0$ is equivalent with
%\begin{equation}
%\label{alpha:quad:2}
 %    \gamma \alpha^{2} - \left( 1 + 3 \left( \gamma - 1 \right) \beta \right) \alpha + \left( 1 - \gamma \right) \beta < 0 .
%\end{equation}
%We have for $1<\gamma<2$
%\begin{align*}
 %   \Delta_{2} & := \left( 1 + 3 \left( \gamma - 1 \right) \beta \right) ^{2} + 4 \gamma \left( \gamma - 1 \right) \beta \\
  %  & > \left( 1 + 3 \left( \gamma - 1 \right) \beta \right) ^{2} > 0 ,
%\end{align*}
%and therefore \eqref{alpha:quad:2} holds for $\alpha>0$ if and %only if
%\begin{equation*}
    %0 < \alpha < \dfrac{1 + 3 \left( \gamma - 1 \right) \beta + %\sqrt{\Delta_{2}}}{2 \gamma} .
%\end{equation*}
%Hence, it remains to observe that \eqref{alpha:1-2} implies
%\begin{align*}
 %   2\gamma\alpha<\:& \left( 2 \gamma + 4 \left( \gamma - 1 \right)\right) \alpha  \\ <\:& 2 - \gamma + 2 \left( 1 - \gamma \right) \beta \\
  %  < \:& 2 + 2 \left( 1 - \gamma \right) \beta \nonumber \\
  %  < \: & 1 + 3 \left( \gamma - 1 \right) \beta + \left\lvert 1 + 3 \left( \gamma - 1 \right) \beta \right\rvert \nonumber \\
  %  < \: & 1 + 3 \left( \gamma - 1 \right) \beta + \sqrt{\Delta_{1}} ,
%\end{align*}
%where we used that $\gamma>1$. 
%It easy to check that \eqref{alpha:0-1} and \eqref{alpha:1-2} imply \eqref{para:beta}.
\end{proof}

% For $0 < \gamma < 1$, then for every $0 < \beta < 1 + \frac{\gamma}{2 \left( 1 - \gamma \right)}$ we have  and thus the first condition becomes
% \begin{equation*}
%     2 \gamma \alpha < 2 - \gamma + 2 \left( \gamma - 1 \right) \beta .
% \end{equation*}
% Now notice that the second inequality is equivalent with
% \begin{multline*}
%     0 < \alpha \\
%     < \dfrac{1 + \left( \gamma - 1 \right) \beta + \sqrt{\left( 1 + \left( \gamma - 1 \right) \beta \right) ^{2} - 4 \gamma \left( \gamma - 1 \right) \beta}}{2 \gamma} .
% \end{multline*}
% Further calculation shows that 
% \begin{multline*}
%     2 - \gamma + 2 \left( \gamma - 1 \right) \beta \\
%     < 1 + \left( \gamma - 1 \right) \beta + \sqrt{\left( 1 + \left( \gamma - 1 \right) \beta \right) ^{2} - 4 \gamma \left( \gamma - 1 \right) \beta} .
% \end{multline*}
% Indeed since , we have

\begin{lem} \label{lem:ABC} Let $A$, $B$ and $C$ be positive real numbers. Then $A\geq \frac{C^2}{B+C}$ implies 
\begin{align*}
    \frac{9}{4}A\geq 2C-B.
\end{align*}
\end{lem}
\begin{proof}
Rewriting the given assumption leads to
\begin{align*}
   \sqrt{ (B+C)A}\geq C.
\end{align*}
We can apply the weighted arithmetic-geometric mean inequality $\frac{2}{3}\alpha+\frac{3}{2} \beta \geq 2\sqrt{\alpha \beta}$ and get
\begin{align*}
    \frac{2}{3}(B+C)+\frac{3}{2}A\geq 2C.
\end{align*}
By rewriting this inequality and multiplying by $\frac{3}{2}$ we get the conclusion.
\end{proof}

\end{document}